\documentclass[a4paper,11pt]{amsart}
\usepackage[utf8]{inputenc}
\usepackage{amsmath,amsfonts,amssymb,anysize,amscd}
\usepackage[all,cmtip]{xy}
\usepackage[onehalfspacing]{setspace}
\usepackage{amsthm}
\usepackage{graphicx}
\usepackage{pstricks-add}
\usepackage{pst-all}
\usepackage{caption}

\theoremstyle{plain}
\newtheorem{thm}{Theorem}[section]

\newtheorem{prop}[thm]{Proposition}

\theoremstyle{definition}
\newtheorem{defn}[thm]{Definition}

\theoremstyle{remark}
\newtheorem{rem}[thm]{Remark}

\def\z{\mathbf{Z}}
\def\cc{{\mathbf G}_{\mathbf m}}

\def\n{\mathbf{N}}

\def\L{\mathcal{L}}

\def\O{\mathcal{O}}

\def\V{\mathcal{V}}

\def\Pic{{\rm Pic}}
\def\P{\mathcal{P}}
\def\W{\mathcal{W}}

\def\Ker{{\rm Ker}}

\def\g{{\mathfrak g}}
\def\IP{{\mathbf P}}

\def\F{\mathcal{F}}
\def\E{\mathcal{E}}
\def\:{\colon}
\def\ox{\otimes}

\def\red{\text{\rm red}}

\title{Continuous linear series}

\author[Esteves]{Eduardo Esteves}
\address{Instituto de Matem\'atica Pura e Aplicada, IMPA, Rio de Janeiro, RJ\\Brazil}
\email{esteves@impa.br}

\author[Nigro]{Antonio Nigro}
\address{Universidade Federal Fluminense, Instituto de Matem\'atica e Estat\'istica, Niter\'oi, RJ, Brazil}
\email{antonio.nigro@gmail.com}

\author[Rizzo]{Pedro  Rizzo}
\address{Instituto de Matemáticas\\Universidad de Antioquia\\ Medell\'in, Colombia}
\email{pedro.hernandez@udea.edu.co }

\keywords{(Limit) Linear series on Curves, Hilbert Schemes, Moduli Problems, Compactification}

\date{}
\begin{document}

\begin{abstract}
We parameterize by a fine moduli space all degenerations of linear series to a singular curve which is the union of two smooth components meeting transversally at a single point. For this we introduce a novel object in the study of degenerations of linear series, which is the \emph{continuous linear series}. Our moduli space can be regarded as a Hilbert quotient, in the terminology introduced by Kapranov, and is a new compactification of Osserman moduli space of exact limit linear series, and consequently, of Eisenbud and Harris moduli space of refined limit linear series on the curve.
\end{abstract}
\subjclass[2020]{14H10, 14C05, 14D22}
\maketitle

\section{Introduction.}
Fix a singular curve $X$ consisting of two smooth components $Y$ and $Z$ meeting transversally at a point $P$. A \emph{generalized linear series} on $X$ is a space of sections of a torsion-free, rank-1 sheaf on $X$. For short, a \emph{continuous linear series} on $X$ is a \emph{maximal primitive} (definitions given later) family of generalized linear series on $X$ parameterized by an Artin stack which is the quotient of a two-punctured semistable chain of rational curves modulo an action of the one-dimensional torus $\cc$. We parameterize them by a projective scheme. 

The history of the study and applications of degenerations of linear series goes a long way back, even predating work by the Italian School of Algebraic Geometry of Severi and others from the turn to the 20th Century. The modern study arose from applications to Brill--Noether Theory of the theory of limit linear series developed by Eisenbud and Harris in the 1980's. 

Degenerating linear series along families of smooth curves degenerating to a singular curve have many limits. For limit curves of compact type, Eisenbud and Harris \cite{EH1} took those limits whose restrictions to the components of the limit curve have degree zero but for a component, the focus, and verified the collection of their restrictions to their respective foci is a \emph{limit linear series}. A limit linear series, according to Eisenbud and Harris \cite{EH1}, pp.~345--6, is a collection of linear series of same rank and degree, one on each component of the limit curve, satisfying certain inequalities of order sequences. It is called \emph{refined} if all the inequalities are equalities. Eisenbud and Harris proceeded to produce a moduli space parameterizing refined limit linear series over curves of compact type, \cite{EH1}, Thm.~3.3, p.~354. 
Limit linear series arising from degenerations may not be refined, so their space is only quasi-projective.

Multiple applications of their theory were found to this day; see \cite{EH2,EH3,EH4,Cot,Tar,Cast,Fark1,Fark2} for a sample. On the other hand, multiple efforts (including by the authors) have not yet produced a natural, convenient and convincing generalization of the theory to all nodal curves; see \cite{Ran,Ee,EM,ES,Oss1,Oss2,ENR,ESVI,ESVII} for a sample. Attempts were even made using Tropical Geometry (the connection explained in \cite{OssT}), which have nonetheless led to new applications; see \cite{BN,AB,CDP,JP1,JP2} for a sample.

In the 2000's Ossermann \cite{Oss1} argued for another approach for curves of compact type, which is more functorial and works in positive characteristic. Instead of taking limits of linear series whose multidegrees were focused on a component, he considered the collection of all the limits that had effective multidegrees, and verified it is an \emph{Osserman limit linear series}. Over our curve $X$, an Osserman limit linear series of rank $r$ and degree $d$ is a collection of $d+1$ linear series $\mathfrak g=\{(L^{(0)},V^{(0)}),\dots,(L^{(d)},V^{(d)})\}$ of rank $r$ and degree $d$, with $\deg(L^{(0)}_{|Y})=d$ and, for each $i$,
\begin{enumerate}
    \item $\big(L^{(i)}_{|Y}\,,\, L^{(i)}_{|Z}\big)=\big(L^{(0)}_{|Y}(-iP)\,,\, L^{(0)}_{|Z}(+iP)\big)$;
    \item $\varphi^i(V^{(i)})\subseteq V^{(i+1)}$ and $\varphi_i(V^{(i+1)})\subseteq V^{(i)}$, where $\varphi^i\colon L^{(i)}\to L^{(i+1)}$ and $\varphi_i\colon L^{(i+1)}\to L^{(i)}$ are the unique (modulo scalar) nontrivial maps.
\end{enumerate}
Osserman \cite{Oss1}, Thm.~5.3, p.~1178, went on to produce a projective moduli space $G^{r,\text{Oss}}_d(X)$ parameterizing those $\mathfrak g$. He worked only over our curve $X$, but later \cite{Oss2} extended the theory to all nodal curves.

In \cite{E-Oss} Osserman and the first author showed another advantage of Osserman's approach. They showed how to produce a collection $\mathbf P(\mathfrak g)$ of degree-$d$ divisors of $X$ out of an Osserman limit linear series $\mathfrak g=\{(L^{(0)},V^{(0)}),\dots,(L^{(d)},V^{(d)})\}$ over $X$, and showed it coincides with the collection of limits of divisors if $\mathfrak g$ arises from a degeneration to $X$, as long as $\mathfrak g$ is \emph{exact}, a term coined by Osserman in \cite{Oss1}, Def.~A10, p.~1197, meaning that the induced sequences 
$$
V^{(i)}\to V^{(i+1)}\to V^{(i)}\to V^{(i+1)}
$$
are exact for all $i$; see \cite{E-Oss}, Def.~4.1, p.~83 and Thm.~5.2, p.~90. (Exactness is also necessary by \cite{ENR}, Prop.~5.4, p.~846.)
This does not hold for Eisenbud and Harris limit linear series, unless they are refined, in which case they are essentially exact Osserman limit linear series; see \cite{Oss1}, Cor.~6.8, p.~1189 and Rmk.~6.16, p.~1192. (This correspondence is not true for curves with more components; see \cite{M}.) But exactness is an open condition. It turned out then that Osserman moduli space $G^{r,\text{Oss}}_d(X)$ is projective, but only an open subset of it --- though dense for general $X$ by Liu \cite{Liu} --- parameterizes enough data to determine the limits of divisors.

Degenerating linear series along a degenerating family of curves do yield at the limit exact Osserman limit linear series, as long as the total space of the family is smooth; see \cite{E-Oss}, \S~5, p.~90. For our curve $X$ as limit curve, smoothness is controlled by a single number $\delta$, called the \emph{singularity degree} at $P$, smoothness corresponding to $\delta=1$. We \cite{ENR} have realized that one could replace a nonexact limit linear series by an exact one, by passing to the semistable reduction of the family, albeit one that generalized Osserman limit linear series, by allowing for generalized linear series and for a number of them in the collection depending on $\delta$. Following Osserman, we \cite{ENR}, Prop.~3.2, p.~834, constructed a projective variety $G^r_{d,\delta}(X)$ parameterizing what we termed \emph{level-$\delta$ limit linear series}; see \cite{ENR}, Def.~3.1, p.~833, or Section~\ref{deltalls} for a generalization. Not all of them are exact, and as before only the exact ones yield the right collection of divisors; see \cite{ENR}, Thms.~5.5~and~5.6, p.~846. So, again, we were stuck with a quasi-projective variety parameterizing enough data.

As promised in \cite{ENR}, p.~829, we are able to put together all the exact loci of the varieties $G^r_{d,\delta}(X)$ for all $\delta$, and produce, under appropriate identifications, a projective variety $G^r_d(X)$; see Theorem~\ref{allevels}. We realized that the collection of linear series in a limit linear series $\mathfrak g$ of any level can be put together as a $\cc$-invariant family of generalized linear series over $X$ parameterized by a two-punctured semistable chain of rational curves $(S,N_0,N_{\infty})$ if and only if $\mathfrak g$ is exact, a fact essential in the proof of our Theorem~\ref{allevels}. The two punctures $N_0$ and $N_{\infty}$ are on the initial and final curve of the chain, and by $\cc$-invariant we mean invariant under an action of $\cc$ on $(S,N_0,N_{\infty})$. Of course, such a family corresponds to one parameterized by the quotient stack $[S-\{N_0,N_{\infty}\}/\cc]$. The latter family is \emph{maximal} if it contains linear series with sections on invertible sheaves of all possible nonnegative degrees on $Y$ (and on $Z$) and \emph{primitive} if the family is not the pullback of a family over a chain with less components. It is interesting to observe that $[S-\{N_0,N_{\infty}\}/\cc]$ is the limit of a point! 

We avoid stacks after this introduction. In this paper, the families we consider, which we do call \emph{continuous linear series}, are defined as $\cc$-invariant families of generalized linear series over two-punctured semistable chains of rational curves; see Definition~\ref{inv-const}. The notion of maximality is embedded in Definition~\ref{tgls} of twisted generalized linear series, and primitiveness is everywhere nonconstancy. Definition~\ref{thefunctor} is that of a functor $\mathfrak{G}_d^r(X)$ parameterizing families of continuous linear series on $X$, and Theorem~\ref{funct-equiv} states that the functor is represented by a projective scheme $G^r_d(X)$. Finally, Theorem~\ref{allevels} states that every point on $G^r_d(X)$ can be viewed as a level-$\delta$ limit linear series for some $\delta$.

In Proposition~\ref{betaclass} we show that the chain of rational curves $S$ parameterizing a continuous linear series can be embedded as a $\cc$-invariant subscheme of a certain fiber $H^r_d(X,L)$ of a $\cc$-invariant map $H^r_d(X)\to J$ constructed in Section~\ref{gls}, where $J$ is a component of the Picard scheme of $X$ and $H^r_d(X)$ is a scheme parameterizing generalized linear series. Furthermore, the fiber $H^r_d(X,L)$ is a $\cc$-invariant subscheme of a product $\text{Grass}\big(r+1,W_1\oplus W_2\big)\times T$, where $T$ is the rational chain of $d+1$ rational curves; see Section~\ref{tG}. Were it not for $T$, we would be in the context of work by Giansiracusa and Wu~\cite{GW22}, who generalized seminal work by Kapranov~\cite{Kap-chow, GKZ}. In fact, stripping a level-$\delta$ limit linear series of its sheaves, we are left with a linked sequence of vector spaces, using the terminology in \cite{ENR}, Def.~3.5, p.~835, which can be regarded as a complete colineation; see \cite{Tha} and \cite{VE}. Thaddeus \cite{Tha} showed how the moduli space of complete colineations can be realized as a Chow, Hilbert or Mumford quotient of the action of $\cc$ on a Grassmannian. Our approach is that of constructing a Hilbert quotient, but instead of simply for the action of $\cc$ on the Grassmannian $\text{Grass}\big(r+1,W_1\oplus W_2\big)$, the action on its product by $T$.

The crux of the paper is the proof of Theorem~\ref{chainmaps},
the longest of the paper. Using it, not only do we prove representability for our functor $\mathfrak{G}_d^r(X)$, thus constructing the fine moduli space of continuous linear series, but we show the space is an open and closed subscheme of the Hilbert scheme of $\cc$-invariant subschemes of appropriate Hilbert polynomial of $H^r_d(X)$, our Theorem~\ref{funct-equiv}. 

It might be possible to generalize the approach in this paper to all nodal curves $X$. The space of embedded sheaves $\mathbf R^{\mathbf J^{\mathbf b}}$ in \cite{AEc}, \S~2, would replace $J\times T$, and a (to-be-defined) scheme $\mathbf H^r_d(X)$ fibered by Grassmannians over $\mathbf R^{\mathbf J^{\mathbf b}}$ parameterizing generalized linear series would replace $H^r_d(X)$. In \cite{AEc}, Thm.~4.7, degenerations of degree-$d$ invertible sheaves are parameterized by certain subschemes $Y^{a,b}_{\ell,\mathfrak m}$ of $\mathbf R^{\mathbf J^{\mathbf b}}$, which must be images of (to-be-defined) subschemes $Z^{a,b}_{\ell,\mathfrak m}$ of $\mathbf H^r_d(X)$ if one considers degenerations of $(r+1)$-dimensional spaces of sections of these sheaves. One should then look into the piece of the Hilbert scheme of $\mathbf H^r_d(X)$ parameterizing the $Z^{a,b}_{\ell,\mathfrak m}$. The difficulty is that instead of chains of rational curves, we have the  $Y^{a,b}_{\ell,\mathfrak m}$, which are more generally unions of toric varieties induced by tilings of an Euclidean space by polytopes, as described in \cite{AEb}, \S~4.2. The corresponding subschemes $Z^{a,b}_{\ell,\mathfrak m}$ of $\mathbf H^r_d(X)$ should be as complicated, hence a theme for another paper. 

A logarithmic approach may be possible as well. It might replace handling the toric varieties in the approach delineated above. Carocci has recently announced at the Isaac Newton Institute joint work with Battistella and Wise which looks like a log version of the approach above. On the other hand, a discrete approach has recently appeared in work by Amini and the first author \cite{AE1}, which led to applications to canonical series \cite{AE2}.  

This work is based on the doctor thesis by the third author \cite{Rizzo}, where an approach using stable maps rather than Hilbert schemes was developed. We would like to thank Omid Amini, Margarida Melo, Brian Osserman, Marco Pacini, Eduardo Vital and Filippo Viviani for helpful discussions on the subject.

\section{Continuous linear series}\label{sec:stablells}

\textbf{Fix} an algebraically closed field $k$. All our schemes will be $k$-schemes. For each scheme $T$, denote by $T(k)$ its set of $k$-points. 

A reduced projective scheme of pure dimension $1$ is called a \emph{curve}. A coherent sheaf $L$ on a curve $X$ is called \emph{torsion-free, rank-one} if its associated points are the generic points of $X$, over which the sheaf has rank 1. Its \emph{degree} is $\chi(L)-\chi(\mathcal O_X)$.

Let $\pi:\mathcal{X}\rightarrow B$ be a flat projective map whose fibers over $k$-points are curves. We call $\pi$ or $\mathcal{X}/B$ a \emph{family of curves}. Given $b\in B(k)$, a subscheme $S\subseteq\mathcal X$ and a coherent sheaf $\mathcal H$ on $\mathcal{X}$, we let $S_b$ denote the intersection of $S$ with the fiber $\mathcal X_b$ of $\pi$ over $b$, and $\mathcal H_b$ the restriction of $\mathcal H$ to the fiber.

A $B$-flat coherent sheaf $\mathcal L$ on $\mathcal{X}$, for short a \emph{sheaf} on $\mathcal X/B$, is called \emph{relatively torsion-free, rank-one on $\mathcal{X}/B$} if its restriction $\mathcal L_b$ to the fiber $\mathcal{X}_b$ is a torsion-free rank-one sheaf on $\mathcal{X}_b$ for each $b\in B(k)$. It is called of (\emph{relative}) \emph{degree} $d$ if $\mathcal L_b$ has degree $d$ for each $b$. An invertible sheaf on $\mathcal X$ is relatively torsion-free, rank-one on $\mathcal{X}/B$.

A \emph{family of generalized linear series on $\mathcal X/B$} is the pair $(\L,\V)$ of a relatively torsion-free, rank-one sheaf $\mathcal L$ on $\mathcal{X}/B$ and a locally free subsheaf 
$\mathcal V\subseteq\pi_*\mathcal L$ such that the composition of the restriction with the base-change map,
$$
\mathcal V_b\to\pi_*(\mathcal L)_b\to H^0(\mathcal X_b,\mathcal L_b),
$$
is injective for each $b\in B(k)$. It is said to have \emph{rank} $r$ if $\mathcal V$ has constant rank $r+1$, and (\emph{relative}) \emph{degree} $d$ if $\mathcal L$ has relative degree $d$. 

If $B=\text{Spec}(k[[t]])$, the data of $\mathcal{X}/B$ and an isomorphism from the closed fiber of $\mathcal{X}/B$ to a curve $X$ is called a {\it regular smoothing of} $X$ if $\mathcal{X}$ is regular and the generic fiber of $\mathcal{X}/B$ is smooth over the field of Laurent series $k\{\{t\}\}$. We will often assume the isomorphism to $X$ is given without mentioning it.

Let ${\mathbf G}_{\mathbf m}:=\text{Spec}(k[t,t^{-1}])$, the one-dimensional torus over $k$. If ${\mathbf G}_{\mathbf m}$ acts on a $k$-scheme $T$, we will let $x\ast t$ denote the action of each $x\in{\mathbf G}_{\mathbf m}(k)$ on 
each $t\in T(k)$.

\textbf{Fix} a curve $X$. \textbf{Assume} $X$ has two irreducible components, denoted $Y$ and $Z$, which intersect at a single point, denoted $P$, and are smooth and intersect transversally, so that $P$ is a node of $X$.

For each $k$-scheme $B$ and integers $i,j$, an invertible sheaf $\mathcal L$ on $X\times B$ will be said to have \emph{relative bidegree $(i,j)$} on $X\times B/B$ if the restrictions of $\mathcal L$ to $Y\times b$ and $Z\times b$ have degrees $i$ and $j$, respectively, for each $b\in B(k)$. 

\subsection{The family of twisters}\label{sec:twisters}

A \emph{chain} (\emph{of rational curves}) over $k$ is a connected curve $T$ whose components are smooth, rational and can be ordered, $T_0,\dots,T_d$, in such a way that $T_i$ intersects $T_j$ only if $|i-j|\leq 1$, and the intersection is transverse, and at a single point. The ordering is unique up to reversing it. The components $T_0$ and $T_d$ are called \emph{extreme}. Given two $k$-points $N_0\in T_0(k)$ and $N_\infty\in T_d(k)$ on the smooth locus of $T$, we call $(T,N_0,N_\infty)$ a \emph{two-pointed chain}. From now on, we will only consider two-pointed chains, and we will refer to them as just chains. The nodes of $T$ and $N_0,N_{\infty}$ are called the \emph{special} points of the chain. For each $i=0,\dots,d$, denote by $T^*_i$ the set of nonspecial points of $T$ on $T_i$ and $T^*=\bigcup T_i^*$. There are actions of $\cc$ on $T$ leaving $T_i$ and the special points of $T$ on each $T_i$ fixed, and making $T_i^*$ a torsor for $i=0,\dots,d$. We call such an action a \emph{chain action}.

\textbf{Fix} a nonnegative integer $d$. \textbf{Fix} a chain $(T,N_0,N_\infty)$ having $d+1$ components. \textbf{Fix} an ordering $T_0,\dots,T_d$ of the components such that $T_i$ intersects $T_j$ only if $|i-j|\leq 1$. For each $i=1,\dots,d$, let
$N_i$ be the intersection of $T_{i-1}$ and $T_{i}$. For convenience, put $N_{d+1}:=N_\infty$. 

For each $i=0,\dots,d$, \textbf{fix} an isomorphism $T_i\xrightarrow{\cong}\IP^1$ such that, denoting by $\xi^i\colon\O_{T_i}^2\to\O_{T_i}(1)$ the tautological quotient of $T_i$, we have that $\xi^i|_{N_i}(\O_{N_i}\oplus 0)=0$ and $\xi^i|_{N_{i+1}}(0\oplus\O_{N_{i+1}})=0$. \textbf{Fix} $E_i\in T_i(k)$ such that $\xi^i|_{E_i}$ is the cokernel of the diagonal $\mathcal O_{E_i}\hookrightarrow\O_{E_i}^2$. There is a unique chain action of $\cc$ on $(T,N_0,N_\infty)$ such that $\xi^i|_{x\ast Q}=\xi^i|_Q(x,1)$, where $(x,1)$ is the natural endomorphism on $\O_{Q}^2$, for each $x\in\cc(k)$ and $Q\in T_i(k)$. \textbf{Fix} this action. Notice that for each $i=0,\dots,d$ and each $Q\in T_i^*(k)$, we have that 
$\lim_{x\to 0}x\ast Q=N_i$, whereas $\lim_{x\to\infty}x\ast Q=N_{i+1}$.

\textbf{Fix} isomorphisms $\O_Y(P)|_P\xrightarrow{\cong}\O_P$ and $\O_Z(P)|_P\xrightarrow{\cong}\O_P$, and from these construct isomorphisms $\O_Y(iP)|_P\cong\O_P$ and $\O_Z(jP)|_P\cong\O_P$ for all integers $i$ and $j$. Using them we obtain a surjection on $X$ for each $i=0,\dots,d$:
\[
\tau^i\colon \O_Y(-iP)\oplus\O_Z(iP) \longrightarrow
\O_Y(-iP)|_P\oplus\O_Z(iP)|_P \xrightarrow{\, \, \cong\, \, } \O_P^2.
\]
Let $\O_X^{(i)}$ be the kernel of $(1,-1)\tau^i$. 
Then $\O_X^{(i)}$ is an invertible sheaf on $X$ whose restriction to $Y$ 
is $\O_Y(-iP)$ and to $Z$ is $\O_Z(iP)$. Let $\tau^i_{T_i}$ be the pullback of
$\tau^i$ and $\xi^i_X$ that of $\xi^i$ to $X\times T_i$.
Identify $P\times T_i=T_i$, and consider
the composition $\xi^i_{X}\tau^i_{T_i}$. It is a surjection of sheaves on
$X\times T_i$; let $\mathcal F_i$ denote its kernel. Identifying $X\times R=X$ for each $R\in T_i(k)$, we see that $\F_i|_{X\times E_i}=\O_X^{(i)}$. Furthermore, 
\[
\mathcal F_i|_{X\times N_i}= \O_Y(-iP)\oplus\O_Z((i-1)P)\quad\text{and}\quad
\mathcal F_i|_{X\times N_{i+1}}=\O_Y(-(i+1)P)\oplus\O_Z(iP).
\]
As $\F_i|_{X\times N_{i+1}}=\F_{i+1}|_{X\times N_{i+1}}$ for $i=0,\dots,d-1$, there is a coherent subsheaf 
\[
\mathcal F\subset p^*_1\big(\O_Y\oplus\O_Z(dP)\big)
\]
with $T$-flat quotient, where $p_1\:X\times T\to X$ is the indicated projection, such that $\mathcal F|_{X\times T_i}=\mathcal F_i$ for each $i=0,\dots,d$. Then $\F$ is a relatively torsion-free sheaf of rank $1$ and degree $0$ on $X\times T/T$. Abstractly, $\F_i|_{X\times T_i^*}\cong\O_X^{(i)}\otimes\O_{T_i^*}$. 

The sheaf $\F$ is invariant under the action of $\cc$: For each $x\in\cc(k)$, each $i=0,\dots,d$ and $Q\in T_i$, we have
\begin{equation}\label{QcQ}
\F|_{X\times(x\ast Q)}=\Ker(\xi^i|_Q(x,1)\tau^i)=(x^{-1},1)\Ker(\xi^i|_Q\tau^i)=(x^{-1},1)\F|_{X\times Q}
\end{equation}
as subsheaves of $\O_Y(-iP)\oplus\O_Z(iP)$.

We say $\mathcal F$ is the \emph{family of twisters} associated to $X$ and $d$, and \textbf{fix} the notation.

\subsection{An alternative construction}

For the sake of completeness we sketch next an alternative construction of the sheaf $\F$ through a degeneration argument. This construction will not be used elsewhere in the paper.

Let $\widehat T$ be the chain obtained from $T$ by
adding an extra rational curve at each end: one,
denoted $T_{-1}$, intersecting $T_0$ transversally
at $N_0$, the other, denoted $T_{d+1}$, intersecting
$T_d$ transversally at $N_{d+1}$. View $T\subset\widehat T$.

Let $B:=\text{Spec}(k[[t]])$. Let $\pi\colon\mathcal{X}\rightarrow B$ (resp.~$\tau:\mathcal{T}\rightarrow B$) be a regular smoothing of $X$ (resp.~$\widehat T$). Such maps exist, as the universal deformation space of a nodal curve is regular. Form the threefold
$\mathcal{X}\times_B\mathcal{T}$. It is a
smoothing of the surface $X\times\widehat T$. However,
it fails to be regular exactly at $(P,N_i)$
for $i=0,\dots, d+1$. Following the ideas in ~\cite[\S~4.2]{CEP}, 
blow up $\mathcal{X}\times_B\mathcal{T}$ successively
along the (strict transforms of)
$Y\times T_{-1}$, $Y\times T_0$, \dots, $Y\times T_d$ in this order.
Denote by $\widetilde{\mathcal X}$ the resulting space,
and by $\mathcal Y_i$ and $\mathcal{Z}_i$ the strict transforms in
$\widetilde{\mathcal X}$ of $Y\times T_i$ and $Z\times T_i$ for each
$i=-1,\dots,d+1$. It follows that $\widetilde{\mathcal X}$
is regular and the $\mathcal Y_i$ and $\mathcal Z_i$ are Cartier divisors of
$\widetilde{\mathcal X}$. Let
$\psi\colon\widetilde{\mathcal X}\to\mathcal{X}\times_B\mathcal{T}$
denote the blowup map. In the sense of ~\cite[Lem.~5.3, p.~2953]{CEP},
the map $\psi$ is a semistable modification of 
each of the two projections of the fibered product
$\mathcal{X}\times_B\mathcal{T}$ onto its factors. In particular, 
the composition of the blowup with any of 
these projections is flat.

Let $C_i:=\psi^{-1}(P,N_i)$ for each $i=0,\dots,d+1$. Then the
$C_i$ are rational smooth curves. 
The strict transforms $\mathcal Y_{i-1}$ and $\mathcal Z_i$ contain $C_i$, 
whereas $\mathcal Y_{i}$ and $\mathcal Z_{i-1}$ intersect $C_i$
transversally at unique and distinct points. We have that 
$(\mathcal Y_j+\mathcal Z_j)\cdot C_i=0$ for each $j=-1,\dots,d+1$. Also, $\mathcal Y_j\cdot C_i=0$ if $j\neq i-1,i$, whereas $\mathcal Y_i\cdot C_i=1$ and $\mathcal Y_{i-1}\cdot C_i=-1$.

Set
$$
\mathcal{G}:=\O_{\widetilde{\mathcal{X}}}
\Big(-\mathcal Y_{-1}+\mathcal Y_1+2\mathcal
Y_2+\cdots+(d+1)\mathcal Y_{d+1}\Big).
$$
Then $\mathcal{G}|_{C_i}\cong\O_{C_i}(1)$ for each
$i=0,\dots,d+1$, and hence the sheaf $\mathcal G$ is $\psi$-admissible, in the sense of \cite[\S~3]{EP}. It follows from ~\cite[Thm. 3.1, p.~63]{EP} that $\psi_{\ast}\mathcal{G}$ is a relatively torsion-free, rank-1 sheaf of degree 0 on
$\mathcal{X}\times_B\mathcal{T}/\mathcal{T}$, whose
formation commutes with base change. The restriction $\psi_{\ast}\mathcal{G}|_{X\times T}$ is isomorphic to $\mathcal F$.

\subsection{Twisted generalized linear series}

\begin{defn}
Let $B$ be a scheme. Define the \emph{category $\mathcal C_B$ of families of chains over $B$} as follows:
\begin{enumerate}
    \item A \emph{family of chains over $B$} is the data of a family of curves $\pi\:\mathcal S\to B$ and sections $\mathcal N_0,\mathcal N_\infty\subset\mathcal S$ of $\pi$ such that $(\mathcal S_b,\mathcal N_{0,b},\mathcal N_{\infty,b})$ is a chain for each $b\in B(k)$. 
    \item A \emph{morphism} from a family of chains $(\pi\colon\mathcal S\to B,\mathcal N_0,\mathcal N_\infty)$ to another, $(\pi'\colon\mathcal S'\to B,\mathcal N'_0,\mathcal N'_\infty)$, is a $B$-morphism $\gamma\:\mathcal S\to\mathcal S'$ such that $\gamma^{-1}(\mathcal N'_i)=\mathcal N_i$ for $i=0,\infty$ and $\gamma_*[\mathcal S_b]=[\mathcal S'_b]$ for each $b\in B(k)$, where $[-]$ is the fundamental class in the Chow group.
\end{enumerate}
\end{defn}

\begin{defn} Let $B$ be a scheme. Let $p_1\: B\times T \to B$ be the projection. Define the \emph{category of families of chain maps to $T$ over $B$} as the category $\mathcal C_B(T)$ of morphisms in $\mathcal C_B$ to $(p_1,B\times N_0,B\times N_{\infty})$. 
\end{defn}

The category $\mathcal C_B(T)$ is the same as the category of maps 
$\gamma=(\gamma_1,\gamma_2)\:\mathcal S\to B\times T$ where
\begin{enumerate}
    \item $(\gamma_1,\gamma_2^{-1}(N_0),\gamma_2^{-1}(N_\infty))$ is a family of chains over $B$,
    \item $\gamma_{2*}[\mathcal S_b]=[T]$ for each $b\in B(k)$,
\end{enumerate}
which we call families of chain maps to $T$ over $B$ as well. In this approach, a morphism from a family of chain maps $\gamma\:\mathcal S\to B\times T$ to another, $\eta\:\mathcal S'\to B\times T$, is simply a $(B\times T)$-morphism $f\:\mathcal S\to\mathcal S'$ such that $f_*[\gamma_1^{-1}(b)]=[\eta_1^{-1}(b)]$ for each $b\in B(k)$.

If $B=\mathrm{Spec}(k)$, a family of chain maps to $T$ over $B$ is simply called a \emph{chain map} to $T$. If $f$ is a morphism from one family of chain maps to $T$ over $B$ to another, we say that the latter is \emph{intermediate} to the former; if $f$ is an isomorphism, we say they are \emph{equivalent}. Given a chain map to $T$, the chain action on the target $(T,N_0,N_\infty)$ lifts to a chain action on the source; the lifting may not be unique.

Let $B$ be a scheme. Let $\gamma=(\gamma_1,\gamma_2)\:\mathcal S\to B\times T$ be a family of chain maps. Let $\L$ be an invertible sheaf on
$X\times B$. Consider the induced maps:
\begin{equation}
\xymatrix{&X\times\mathcal S\ar[dl]_{(1,\gamma_1)}\ar[d]_{p_{\mathcal S}} \ar[dr]^{(1,\gamma_2)}&\\
X\times B & \mathcal S & X\times T,}
\end{equation}
where $p_{\mathcal S}$ is the projection. We will call the sheaf
$$
\L(\gamma):=\L\boxtimes\F:=(1,\gamma_1)^*\L\otimes(1,\gamma_2)^*\F
$$
on $X\times\mathcal S$ the \emph{family of twists of $\L$ along} $\gamma$. Notice that $\L(\gamma)=(1,\gamma)^*\L(1_{B\times T})$.

Notice that an isomorphism $g\:\L\to\L'$ induces an isomorphism $g(\gamma)\:\L(\gamma)\to\L'(\gamma)$ in the natural way. 

\begin{defn}\label{tgls} Let $\gamma\:\mathcal S\to B\times T$ be a family of chain maps. A \emph{family of twisted generalized linear series of degree $d$ and rank $r$ on $X$ along $\gamma$} is a pair $(\L,\V)$ consisting of:
\begin{enumerate}
\item an invertible sheaf $\L$ on $X\times B$ of relative bidegree $(d,0)$ on $X\times B/B$;
\item a subsheaf $\V\subseteq p_{\mathcal S\ast}\L(\gamma)$ such that $(\L(\gamma),\V)$ is a family of generalized linear series of rank $r$ on $X\times\mathcal S/\mathcal S$.
\end{enumerate}
When $B=\text{Spec}(k)$, we say that $(\L,\V)$ is a \emph{twisted generalized linear series} on $X$ along $\gamma$.
\end{defn}

Let $L$ be an invertible sheaf on $X$ of bidegree $(d,0)$. Let $\gamma\:S\to T$ and $\gamma'\:S'\to T$ be chain maps. Let $(L,\V)$ (resp.~$(L,\V')$) be a twisted generalized linear series on $X$ along $\gamma$ (resp.~$\gamma'$). Since $\F$ is a subsheaf of $(\O_Y\oplus\O_Z(dP))\otimes\O_T$ with $T$-flat quotient, it follows that 
\[
L(\gamma),L(\gamma')\subseteq \big(L|_Y\oplus L|_Z(dP)\big)\otimes\O_{S}
\]
with $S$-flat quotients, and hence 
\[
\V,\V'\subseteq \big(\Gamma(Y,L|_Y)\oplus\Gamma(Z,L|_Z(dP))\big)\otimes\O_{S}
\]
with $S$-flat quotients. We can thus compare the vector spaces $\V_{s}$ and $\V'_{s'}$ for $s\in S(k)$ and $s'\in S(k')$, and in particular, ask whether they are equal or not; see below. 

\begin{defn}\label{inv-const} Let $(L,\V)$ be a twisted generalized linear series on $X$ along a chain map $\gamma\: S\to T$. Let $C$ be an irreducible component of $S$. We say $(L,\V)$ is \emph{constant} along $C$ if $\gamma|_C$ is constant, and $\V_{s_1}=\V_{s_2}$ for all $s_1,s_2\in C(k)$. It is said to be \emph{invariant} along $C$ if the chain action on $(T,N_0,N_\infty)$ lifts to a chain action on $(S,\gamma^{-1}(N_0),\gamma^{-1}(N_\infty))$ such that $\V_{x\ast Q}=(x^{-1},1)\V_{Q}$ for each $x\in\cc(k)$ and $Q\in C^*(k)$. We say $(L,\V)$ is \emph{everywhere nonconstant} (resp.~\emph{everywhere invariant}) if it is not constant (resp.~it is invariant) along each irreducible component of $S$. It is said to be \emph{continuous} if it is everywhere nonconstant and invariant; in this case, we call $(L,\V)$ a \emph{continuous linear series}.
\end{defn}

Let $(\L,\mathcal{V})$ be a family of twisted generalized linear series on $X$ along a family of chain maps $\gamma\:\mathcal S\to B\times T$. Let $\gamma'\:\mathcal S'\to B\times T$ be another family of chain maps and $f\colon\mathcal S'\to\mathcal S$ a morphism from $\gamma'$ to $\gamma$. Then $(\L,f^*\V)$ is a family of twisted generalized linear series on $X$ along $\gamma'$, called the \emph{expansion} of $(\L,\V)$ along $f$. 

Also, for each map $b\:B'\to B$, considering the Cartesian diagram
\[
\begin{CD}
    \mathcal S'@>b'>> \mathcal S\\
    @V\gamma'VV @V\gamma VV\\
    B'\times T @>b\times 1>> B\times T,
\end{CD}
\]
we have that $\gamma'$ is a family of chain maps, and 
\[
b^*(\L,\V):=((1\times b)^*\L,(b')^*\V)
\]
is a family of twisted generalized linear series on $X$ along $\gamma'$ called the \emph{pullback} of $(\L,\V)$ under $b$. We call $(\L,\mathcal V)$ a \emph{family of everywhere nonconstant} (resp.~\emph{everywhere invariant}, resp.~\emph{continuous}) \emph{twisted generalized linear series on $X$ along $\gamma$} if $b^*(\L,\V)$ is an everywhere nonconstant (resp.~everywhere invariant, resp.~continuous) twisted generalized linear series for each $k$-point $b\:\text{Spec}(k)\to B$.

Finally, given an invertible sheaf $M$ on $B$, we have that 
\[
(\L,\V)\otimes M:=(\L\otimes p_2^*M,\V\otimes\gamma_1^*M)
\]
is another family of twisted generalized linear series on $X$ along $\gamma$, where $p_2\:X\times B\to B$ is the indicated projection.

\begin{defn}\label{tglsequiv} Two families of twisted generalized linear series $(\L,\V)$ and $(\L',\V')$ along families of chain maps $\gamma\colon\mathcal S\to B\times T$ and $\gamma'\colon\mathcal S'\to B\times T$ are said to be \emph{equivalent} if there are an isomorphism $f\colon\mathcal S'\to\mathcal S$ from $\gamma'$ to $\gamma$, an invertible sheaf $M$ on $B$ and an isomorphism $g\:\L'\xrightarrow{\cong}\L\otimes p_2^*M$ such that the induced isomorphism $g(\gamma')$ takes $\V'$ to $f^*\V\otimes(\gamma'_1)^*M$.
\end{defn}

Pullbacks and expansions of equivalent families are equivalent families.

\begin{defn}\label{thefunctor} The (contravariant) functor of
continuous linear series of degree $d$ and
rank $r$ on $X$,
$$
\mathfrak{G}_d^r(X)\colon (\text{Schemes})
\longrightarrow(\text{Sets}),
$$
is the functor that associates to each scheme $B$ the
set $\mathfrak{G}_d^r(X)(B)$ of equivalence classes of families of 
continuous (twisted generalized) linear series of degree $d$ and rank $r$ on $X$ along families of chain maps to $T$ over $B$.
\end{defn}

\section{Parametrization of generalized linear series}\label{gls}

\subsection{The scheme of generalized linear series}\label{gls-sch}

\textbf{Fix} $J:=\Pic^{(d,0)}_X$, the component of the Picard scheme of $X$ parameterizing invertible sheaves of bidegree $(d,0)$, that is, degree $d$ on $Y$ and $0$ on $Z$. Let $\P$ be a Poincar\'e sheaf on $X\times J$. Let $D$ be an effective divisor of $X$ supported away from $P$ and denote by $D'$ its pullback to $X\times J\times T$ under the
projection. Put 
\[
\P\boxtimes\F:=p_{1,2}^*(\P)\otimes p_{1,3}^*(\F)\quad\text{and}\quad
\mathcal E:=p_{2,3\ast}\left(\P\boxtimes\F\otimes\mathcal O(D')\right),
\]
where $p_{1,2}$, $p_{1,3}$ and $p_{2,3}$ are the projections onto the indicated factors:
\begin{equation}\label{Diag:H}
\xymatrix{
&X\times J\times T\ar[dl]_{p_{1,3}}\ar[d]_{p_{2,3}}\ar[dr]^{p_{1,2}}&\\
X\times T & J\times T & X\times J.}
\end{equation}
Assume $D$ has high enough degree, so that $R^1p_{2,3\ast}\left(\P\boxtimes\F\otimes\mathcal O(D')\right)=0$, and hence $\mathcal E$ is locally free. Put 
$\mathcal G:=p_{2,3\ast}(\P\boxtimes\F\otimes\mathcal O(D')|_{D'})$. Then $\mathcal G$ is locally free. Also, restriction induces a map $\alpha\colon\mathcal E\to\mathcal G$. 

Let $r$ be a nonnegative integer. Let $\alpha'\colon\mathcal E'\to\mathcal G'$ be the pullback of $\alpha$ to 
$\text{Grass}_{J\times T}(r+1,\E)$ and $\mathcal V'\subseteq\mathcal E'$ the universal locally free subsheaf of rank $r+1$. Denote by $H^{r}_d(X)$ the degeneracy scheme of $\alpha'|_{\mathcal V'}$, the locus where the map of bundles is zero. \textbf{Fix} the notation. Then $H^r_d(X)$ is a projective scheme, called the \emph{scheme of generalized linear series}.  It is a scheme over $J\times T$.

Let $\nu=(\nu_1,\nu_2)\: H^r_d(X)\to J\times T$ be the induced map. Let $\mathcal V$ be the restriction of $\mathcal V'$ to $H^r_d(X)$. It is a locally free subsheaf of rank $r+1$ of $p_*((1\times\nu)^*(\P\boxtimes\F))$, where $p:X\times H^{r}_d(X)\rightarrow H^{r}_d(X)$ is the projection. Furthermore, the pair 
$((1\times\nu)^*(\P\boxtimes\F),\mathcal V)$ is a family of generalized linear series of degree $d$ and rank $r$ on $X\times H^r_d(X)/H^r_d(X)$.

For further use, we denote by $H^r_d(X,L)$ the fiber of $\nu_1$ over each $L\in J(k)$, and call it the \emph{scheme of generalized linear series with sections in $L$}.  


\subsection{The action.}\label{tG}

Since $J$ is isomorphic to an Abelian variety, every action of $\cc$ on $J$ is trivial. The action of $\cc$ on $T$ lifts thus to a unique action on $J\times T$. 

We claim the action lifts naturally
to one on $H^{r}_d(X)$. Indeed, $H^r_d(X)\subseteq \text{Grass}_{J\times T}(r+1,\E)$. Now,
$$
\P\boxtimes\F\subseteq p_{1,2}^*\Big(\P\otimes p_1^*\big(\mathcal O_Y\oplus\mathcal O_Z(dP)\big)\Big),
$$
where $p_1\colon X\times J\to X$ is the projection map. Put
$$
\W_1:=p_{2\ast}\big(\P\otimes p_1^*\mathcal O_Y(D\cap Y)\big)\text{ and }
\W_2:=p_{2\ast}\big(\P\otimes p_1^*\mathcal O_Z(dP+D\cap Z)\big),
$$
where $p_2\colon X\times J\to J$ is the projection. Then $\mathcal E\subseteq q_1^*(\W_1\oplus\W_2)$, where $q_1\colon J\times T\to J$ is the projection. It follows that 
$$
H^r_d(X)\subseteq \text{Grass}_{J}\big(r+1,\W_1\oplus\W_2\big)\times T.
$$
There is now a natural action of $\cc$ on $\text{Grass}_{J}\big(r+1,\W_1\oplus\W_2\big)$, induced by the automorphism $(x^{-1},1)$ on $\W_1\oplus\W_2$ for each $x\in\cc(k)$. Pair it with the action on $T$. Then $H^r_d(X)$ is invariant and the induced map $H^r_d(X)\to T$ is equivariant. 

We linearize the action on $H^r_d(X)$ as follows. 
First, $T\subseteq(\IP^1)^{d+1}$ with equations $x_iy_j=0$ for $i<j$, where $(x_i,y_i)$ are the coordinates of the $i$-th $\IP^1$ for $i=0,\dots,d$. The embedding is equivariant, for the product action of 
$\cc$ on $(\IP^1)^{d+1}$ given by $x\ast (x_i:y_i)=(xx_i:y_i)$ for each $x\in\cc(k)$ and $i=0,\dots,d$.

Second, consider the Pl\"ucker embedding 
$\text{Grass}_{J}\big(r+1,\W_1\oplus\W_2\big)\subseteq\IP_J(\W)$, where
$$
\W:=\bigwedge^{r+1}(\W_1\oplus\W_2)=\bigoplus_{i=0}^{r+1}\Big(\bigwedge^{i}\W_1\otimes\bigwedge^{r+1-i}\W_2\Big);
$$
it is naturally equivariant. Let $\mathcal A$ be a very ample invertible sheaf on $J$, enough ample that $\W\otimes\mathcal A$ be globally generated. Then 
$\IP_J(\W)\subseteq\IP(\Gamma(J,\W\otimes\mathcal A)^*)\times J$, again naturally equivariant. Furthermore, since $\mathcal A$ is very ample, it induces an embedding $J\subseteq\IP(\Gamma(J,\mathcal A)^*)$. 

Putting everything together, we have an equivariant embedding
$$
\Lambda\colon H^r_d(X)\longrightarrow\IP(\Gamma(J,\W\otimes\mathcal A)^*)\times\IP(\Gamma(J,\mathcal A)^*)\times(\IP^1)^{d+1},
$$
which composes with the Segre embedding to give an equivariant embedding of $H^r_d(X)$ in a projective space $\mathbf P^N$.

The action on $\mathbf P^N$ induces actions on the Hilbert schemes $\text{Hilb}_{H^r_d(X)}$ and $\text{Hilb}_{\mathbf P^N}$. We will denote by $\text{Hilb}^{\phi}_{H^r_d(X)}$ the open (and closed) subscheme of $\text{Hilb}_{H^r_d(X)}$ parameterizing subschemes of $H^r_d(X)$ with multivariate Hilbert polynomial 
$$
\phi(u,v,s_0,\dots,s_d):=u(r+1)+\sum_{i=0}^d s_i+1
$$
with respect to the embedding $\Lambda$. We will also denote by $\text{Hilb}^{\phi,\cc}_{H^r_d(X)}$ and $\text{Hilb}^{\cc}_{\mathbf P^N}$ the closed subschemes of 
$\text{Hilb}^{\phi}_{H^r_d(X)}$ and $\text{Hilb}_{\mathbf P^N}$ parameterizing invariant subschemes of 
$H^r_d(X)$ and of $\mathbf P^N$, respectively. They fit in the natural Cartesian diagram:
$$
\begin{CD}
\text{Hilb}^{\phi,\cc}_{H^r_d(X)} @>>> \text{Hilb}^{\cc}_{\mathbf P^N}\\
@VVV @VVV\\
\text{Hilb}^{\phi}_{H^r_d(X)} @>>> \text{Hilb}_{\mathbf P^N}.
\end{CD}
$$
Both $\text{Hilb}^{\phi,\cc}_{H^r_d(X)}$ and $\text{Hilb}^{\cc}_{\mathbf P^N}$ are examples of \textit{invariant Hilbert schemes}. A detailed study of these schemes and their properties can be found in \cite{Br}.

\subsection{Invariant curves in the Grassmannian}\label{sec:invt-Grass}

Let $W_1$ and $W_2$ be finite-dimensional vector spaces over $k$, and put $W:=W_1\oplus W_2$. Let $\iota_i\colon W_i\to W$ be the natural injection and $\rho_i\colon W\to W_i$ be the natural surjection for $i=1,2$.

Let $n$ be a positive integer and $G:=\text{Grass}\left(n,W\right)$. 
Recall that $G$ is a nonsingular projective variety
with $\text{Pic}(G)=\z$. The degree of a curve in $G$ is its intersection number with the ample generator of $\text{Pic}(G)$.

The action of $\cc$ on $W=W_1\oplus W_2$, given by $x\ast (u_1,u_2):=(x^{-1}u_1,u_2)$ for each $x\in\cc(k)$ and $u_i\in W_i$ for $i=1,2$, induces an action on $G$. For each $x\in\cc(k)$ and $V\in G(k)$, we have $x\ast V=(x^{-1},1)V$. 

For each $V\in G(k)$, let $h_V\colon\mathbb\IP^1\to G$ be the morphism sending $(x:1)$ to $x\ast V$ for each $x\in\cc(k)$. 

\begin{prop}\label{hV} Let $V\in G(k)$ be a nonfixed point. 
Let $V_{\infty,1}:=\iota_1^{-1}(V)$ and 
$V_{0,2}:=\iota_2^{-1}(V)$. Let 
$V_{\infty,2}:=\rho_2(V)$ and $V_{0,1}:=\rho_1(V)$. Then $V_{\infty,1}\subsetneqq V_{0,1}$ and $V_{0,2}\subsetneqq V_{\infty,2}$. Furthermore,
\begin{equation}\label{limlim}
\lim_{x\to 0}x*V=V_{0,1}\oplus V_{0,2}\quad \text{and} \quad
\lim_{x\to \infty}x*V=V_{\infty,1}\oplus V_{\infty,2}.
\end{equation}
Finally, $h_V$ is an isomorphism onto its image $\Gamma$, and
$$
\deg \Gamma=
\dim_k V_{0,1}-\dim_k V_{\infty,1}=
\dim_k V_{\infty,2}-\dim_k V_{0,2}.
$$
\end{prop}

\begin{proof} The first statement is clear, where the inequalities follow from the fact that $V$ is nonfixed. As for the remaining statements, let $V'\subseteq V$ such that 
$V=V_{\infty,1}\oplus V'\oplus V_{0,2}$. Then the projection map $V'\to W_2$ is an isomorphism to a subspace $V'_2\subseteq V_{\infty,2}$ such that 
$V_{\infty,2}=V'_2\oplus V_{0,2}$. Similarly, the projection map $V'\to W_1$ is an isomorphism to a subspace $V'_1\subseteq V_{0,1}$ such that 
$V_{0,1}=V_{\infty,1}\oplus V'_1$. Choose bases $B_{\infty,1}$ of $V_{\infty,1}$, $B'$ of $V'$ and $B_{0,2}$ of $V_{0,2}$. Consider the corresponding bases 
$B'_1$ of $V'_1$ and $B'_2$ of $V'_2$, images of $B'$ under the respective isomorphisms $V'\to V'_1$ and $V'\to V'_2$. Complete $B_{\infty,1}\cup B'_1$ to a basis $B_1$ of 
$W_1$ and $B_{0,2}\cup B'_2$ to a basis $B_2$ of $W_2$. The union $B:=B_1\cup B_2$ is a basis of $W$.

For each subset $I$ of $B$, let $n_{I,1}$ be the number of vectors in $B_1$
and $n_{I,2}$ the number of those in $B_2$. For $I$ with $|I|=n$, let $p_I$ be the corresponding Pl\"ucker coordinate of
$V$ and $p_I(y,z)$ that of $V_{(y:z)}$, the image of $(y:z)$ under $h_V$, for each $(y:z)\in\IP^1(k)$ with $yz\neq 0$.
Then $p_I(y,z)=z^{n_{I,1}}y^{n_{I,2}}p_I$.

It follows from our choice of bases that $p_I\neq 0$ if and only if 
$$
B_{\infty,1}\cup B_{0,2}\subseteq I\subseteq B_{\infty,1}\cup B'_1\cup B'_2\cup B_{0,2}
$$
and the intersections $I\cap B'_1$ and $I\cap B'_2$ are identified in $B'$ with disjoint subsets covering $B'$. It thus follows that $\{(n_{I,1},n_{I,2})\,|\,p_I\neq 0\}$ is equal to
$$
\{(\dim_k V_{\infty,1},\dim_k V_{\infty,2}),(\dim_k V_{\infty,1}+1,\dim_k V_{\infty,2}-1),\dots,(\dim_k V_{0,1},\dim_k V_{0,2})\},
$$
and hence that $h_V$ is an isomorphism to its image $\Gamma$, which has degree
$$
n-\dim_k V_{\infty,1}-\dim_k V_{0,2}=\dim_k V_{\infty,2}-\dim_k V_{0,2},
$$ 
proving the third statement of the proposition.

The second statement follows as well, as it is clear that 
$p_I(y,z)/z^{\dim_k V_{\infty,1}}y^{\dim_k V_{0,2}}$ is zero for $(y:z)=(0:1)$ for all $I$ but $I=B_{\infty,1}\cup B_{0,2}\cup B'_2$, which corresponds to the subspace $V_{\infty,1}\oplus V_{\infty,2}$ of $W$, yielding the second formula in \eqref{limlim}. The first follows analogously.
\end{proof}

\begin{prop} \label{hVV} Let $V,V'\in G(k)$. Assume that $\rho_1(V')=\iota^{-1}_1(V)$ (resp.~$\rho_2(V')=\iota_2^{-1}(V)$). Then the intersection $h_{V}(\IP^1)\,\cap\,h_{V'}(\IP^1)$ is empty or scheme-theoretically a point, the latter if and only if $\rho_2(V)=\iota_2^{-1}(V')$ (resp.~$\rho_1(V)=\iota_1^{-1}(V')$).
\end{prop}

\begin{proof} By the symmetry, we may assume $\rho_1(V')=\iota^{-1}_1(V)$. Now, 
$h_V(\IP^1)$ (resp.~$h_{V'}(\IP^1)$) is contained in the subGrassmannian of $G$ parameterizing $n$-dimensional subspaces of $\rho_1(V)\oplus\rho_2(V)$ (resp.~$\rho_1(V')\oplus\rho_2(V')$).
Since $\rho_1(V')=\iota^{-1}_1(V)$, the intersection of the two subGrassmannians is contained in the Grassmannian of $n$-dimensional subspaces of 
\begin{equation}\label{eq:VV'-intersection}
\iota^{-1}_1(V)\oplus(\rho_2(V)\cap\rho_2(V')).
\end{equation}
But this space has dimension at most $n$, and thus the Grassmannian is either empty or a point. It follows that $h_{V}(\IP^1)\,\cap\,h_{V'}(\IP^1)$ is either empty or scheme-theoretically a point. 

Assume the latter. Then, the point of intersection is the space \eqref{eq:VV'-intersection}, which has dimension $n$ only if $\rho_2(V)\subseteq\rho_2(V')$, in which case \eqref{eq:VV'-intersection} is equal to $\iota^{-1}_1(V)\oplus\rho_2(V)$. This space is clearly a point on $h_V(\IP^1)$. It will be one on $h_{V'}(\IP^1)$ if and only if $\rho_2(V)=\iota_2^{-1}(V')$. \end{proof}

\section{Representability}

\subsection{Embeddedness}

Let $\mathfrak g:=(\L,\V)$ be a family of twisted generalized linear series of degree $d$ and rank $r$ 
along $\gamma$, for a family of chain maps $\gamma=(\gamma_1,\gamma_2)\colon\mathcal S\to B\times T$. Recall $\mathcal P$ from Section~\ref{gls-sch}. Then 
there are an invertible sheaf $N$ on $B$ and a map $h\colon B\to J$ 
such that $(1_X\times h)^*\mathcal P\cong \L\otimes N$. Up to replacing $\mathfrak g$ by an equivalent family, we may assume $N=\mathcal O_B$. 
The induced map $f:=(h\gamma_1,\gamma_2)\colon\mathcal S\to J\times T$ satisfies 
$(1_X\times f)^*(\mathcal P\boxtimes\mathcal F)\cong\L(\gamma)$. The locally free subsheaf $\V\subseteq p_{\mathcal S*}\mathcal L(\gamma)$ 
of rank $r+1$ induces thus a map 
$f_{\mathfrak g}\colon\mathcal S\to H^r_d(X)\times B$.

\begin{prop}\label{betaclass} If $\mathfrak g$ is a family of continuous linear series, then $f_{\mathfrak g}$ is a closed embedding. Furthermore, its image $\mathcal S_{\mathfrak g}$ is $\cc$-invariant and has relative multivariate Hilbert polynomial $\phi$ over $B$. 
\end{prop}

\begin{proof} We may assume $B=\text{Spec}(k)$. Then $f_{\mathfrak g}$ factors through the fiber $H^r_d(X,\L)$ of $H^r_d(X)$ over $\L$, which is naturally an invariant subscheme of $G\times T$, where \[
G:=\text{Grass}(r+1,W)\quad\text{for }
W:=W_1\oplus W_2:=\Gamma(Y,\L|_Y)\oplus\Gamma(Z,\L|_Z(dP)).
\]
Let $\iota_i\colon W_i\to W$ be the natural injection and $\rho_i\colon W\to W_i$ be the natural surjection for $i=1,2$.

As a map to $G\times T$, we have $f_{\mathfrak g}=(f,\gamma_2)$. 
Let $S_0,\dots,S_m$ be the sequence of irreducible components of $\mathcal S$, ordered in such a way that 
\begin{enumerate}
    \item $\gamma_2(S_0)=T_0$ and $\gamma_2(S_m)=T_d$,
    \item $S_i\cap S_j\neq\emptyset$ if and only if $|i-j|\leq 1$.
    \end{enumerate}
Let $Q_i:=S_{i-1}\cap S_i$ for $i=1,\dots,m$, and $Q_0,Q_{m+1}\in\mathcal S$ with $\gamma_2(Q_0)=N_0$ and $\gamma_2(Q_{m+1})=N_\infty$. 

Since $\mathfrak g$ is everywhere invariant, the image under $f_{\mathfrak g}$ of each $S_i$ is invariant, thus $\mathcal S_{\mathfrak g}$ is invariant. We will argue now that $f_{\mathfrak g}$ is an embedding. First of all, since $\mathfrak g$ is also everywhere nonconstant, for each $i=0,\dots,m$, either $\gamma_2|_{S_i}$ is an embedding, or there is $s\in S_i$ such that $f(s)$ is not fixed in $G$, and hence $f|_{S_i}$ is an embedding by Proposition~\ref{hV}. In any case, 
$f_{\mathfrak g}|_{S_i}$ is an embedding.

To show that $f_{\mathfrak g}$ is injective, by contradiction, let $R_1\in S_i$ and $R_2\in S_j$ for $i\leq j$ be distinct points on $\mathcal S$ such that $f_{\mathfrak g}(R_1)=f_{\mathfrak g}(R_2)$. Then $\gamma_2(R_1)=\gamma_2(R_2)$. Since $\gamma_2$ is a chain map, $\gamma_2(S_i)=\gamma_2(S_{i+1})=\cdots=\gamma_2(S_j)=N_\ell$ for some $\ell$. As proved above, $f|_{S_i},f|_{S_{i+1}},\dots,f|_{S_j}$ are embeddings. Furthermore, Proposition~\ref{hV} yields as well that 
\[
\rho_2(f(R_1))=\rho_2(f(Q_{i+1}))\subsetneqq\rho_2(f(Q_{i+2}))\subsetneqq\cdots\subsetneqq\rho_2(f(Q_{j}))\subseteq\rho_2(f(R_2)),
\]
where the last inclusion is strict if $R_2\neq Q_j$, and
\[
\rho_1(f(R_1))\supseteq\rho_1(f(Q_{i+1}))\supsetneqq\rho_1(f(Q_{i+2}))\supsetneqq\cdots\supsetneqq\rho_1(f(Q_{j}))=\rho_1(f(R_2)),
\]
where the first inclusion is strict if $R_1\neq Q_{i+1}$. Since $f(R_1)=f(R_2)$, we must have that either $j=i+1$ and $R_1=Q_j=R_2$, or $j=i$, and thus $R_1=R_2$ because $f|_{S_i}$ is injective. In any case, $R_1=R_2$, a contradiction.

We need now only prove that the differentials $d_{Q_j}f_{\mathfrak g}|_{S_{j-1}}$ and $d_{Q_j}f_{\mathfrak g}|_{S_j}$ at $Q_j$ have different images for $j=1,\dots,m$. 

First, their images are nonzero because $f_{\mathfrak g}|_{S_{j-1}}$ and $f_{\mathfrak g}|_{S_{j}}$ are embeddings. Second, their images are distinct, if $\gamma_2|_{S_{j-1}}$ or $\gamma_2|_{S_j}$ is not constant, because then the images of $d_{Q_j}\gamma_2|_{S_{j-1}}$ and $d_{Q_j}\gamma_2|_{S_j}$ are distinct. 

Assume now that $\gamma_2(S_{j-1})$ and $\gamma_2(S_j)$ are each a point, necessarily the same point, a node $N_\ell$ for some $\ell$. Then $f$ restricts to an injection on $S_{j-1}\cup S_j$, and to embeddings $f|_{S_{j-1}}$ and $f|_{S_j}$. We need only show now that the intersection at the unique point on $f(S_{j-1})\,\cap\,f(S_j)$ is transverse. Now, since $\g$ is everywhere invariant, there are isomorphisms $\theta_1\:\cc\to S_{j-1}^*$ and $\theta_2\:\cc\to S_j^*$ such that $\V_{\theta_i(x)}=(x^{-1},1)V_i$ for each $x\in\cc(k)$ and $i=1,2$, where $V_i:=\V_{\theta_i(1)}$. Also, $f\theta_i$ extends to the map $h_i:=h_{V_i}\colon\IP^1\to G$ defined in Subsection~\ref{sec:invt-Grass}. Then $h_1(\infty)$ and $h_2(0)$ are equal to the subspace $V:=\V_{Q_j}\subseteq W$. By Proposition~\ref{hV}, 
\[
h_1(\infty)=\lim_{x\to\infty}x\ast V_1=\iota_1^{-1}(V_1)\oplus\rho_2(V_1)\quad\text{and}\quad
h_2(0)=\lim_{x\to 0}x\ast V_2=\rho_1(V_2)\oplus\iota_2^{-1}(V_2),
\]
so $h_1(\infty)=h_2(0)$ means
\[
\iota_1^{-1}(V_1)=\rho_1(V_2)\quad\text{and}\quad \rho_2(V_1)=\iota_2^{-1}(V_2).
\]
But then Proposition~\ref{hVV} yields that $h_1(\IP^1)$ and $h_2(\IP^1)$, which are $f(S_{j-1})$ and $f(S_{j})$, intersect transversally. 

It remains to show that $\mathcal S_{\mathfrak g}$ has multivariate Hilbert polynomial $\phi$. First of all, $\mathcal S_{\mathfrak g}$ is a curve, so its multivariate Hilbert polynomial $p_{\mathcal S_{\mathfrak g}}(u,v,s_0,\dots,s_d)$ is linear. Also, $\mathcal S_{\mathfrak g}$ is connected and of arithmetic genus zero, thus the constant term of $p_{\mathcal S_{\mathfrak g}}$ is $1$. Since 
$\mathcal S_{\mathfrak g}\subseteq H^r_d(X,\mathcal L)$, the coefficient of $v$ in $p_{\mathcal S_{\mathfrak g}}$ is zero. Since $\gamma_{2,*}[\mathcal S]=[T]$, the coefficient of $s_i$ is $1$ for each $i=0,\dots,d$. It remains to show that the coefficient of $u$ is $r+1$, which is equivalent to showing that $f(\mathcal S)$ has degree $r+1$ in $G$.

Let $U_j:=\V_{Q_j}\subseteq W$ for $j=0,1,\dots,m+1$. It follows from Proposition~\ref{hV} that the degree of $f(\mathcal S)$ is 
$$
\sum_{i=1}^{m+1}\Big(\dim_k\rho_1(U_{i-1})-\dim_k\rho_1(U_{i})\Big),
$$
which is equal to $\dim_k\rho_1(U_0)-\dim_k\rho_1(U_{m+1})$.
But 
$$
U_0\subseteq\Gamma(Y,L|_Y)\oplus\Gamma(Z,L|_Z(-P))\quad\text{and}\quad
U_{m+1}\subseteq\Gamma(Y,L|_Y(-(d+1)P))\oplus\Gamma(Z,L|_Z(dP)).
$$
Since $L|_Z(-P)$ and $L|_Y(-(d+1)P)$ have negative degrees, $\dim_k\rho_1(U_0)=r+1$ and $\rho_1(U_{m+1})=0$. 
Thus, the degree of $f(\mathcal S)$ is $r+1$. 
\end{proof}

\subsection{First main theorem}

\begin{thm}\label{chainmaps} For each connected $S\in \text{\rm Hilb}^{\phi,\cc}_{H^r_d(X)}(k)$, the induced map $S\to T$ is a chain map.
\end{thm}

The proof of this theorem is postponed to the next section. We use it now to prove the first main theorem of the article.

\begin{thm}\label{funct-equiv}
The functor $\mathfrak{G}_d^r(X)$ is isomorphic to the functor of points of an open and closed subscheme of $\text{\rm Hilb}_{H_d^r(X)}^{\phi,\cc}$. In particular, there exists a projective scheme $G_d^r(X)$ representing $\mathfrak{G}_d^r(X)$.
\end{thm}

\begin{proof} Proposition~\ref{betaclass} associates to a family of continuous linear series $\mathfrak g=(\L,\V)$ over a scheme $B$ a certain family of $\cc$-invariant subschemes $\mathcal S_{\mathfrak g}\subseteq H^r_d(X)\times B$ over $B$ with relative multivariate Hilbert polynomial $\phi$ 
whose fibers satisfy $h^0(\mathcal S_{\mathfrak g,b},\mathcal O_{\mathcal S_{\mathfrak g,b}})=1$ for every $b\in B$, thus a 
$B$-point of the open subscheme $U\subseteq\text{\rm Hilb}_{H_d^r(X)}^{\phi,\cc}$ parameterizing subschemes $S$ with $h^0(S,\mathcal O_S)=1$. If $\mathfrak g=(\L',\V')$ is another family equivalent to $\mathfrak G$, the corresponding subscheme $\mathcal S_{\mathfrak g'}$ is equal to $\mathcal S_{\mathfrak g}$, as it follows from Definition~\ref{tglsequiv} and the moduli properties of Jacobians and Grassmannians. 

Conversely, let $\mathcal S\subseteq H_d^r(X)\times B$ be a family of 
$\cc$-invariant closed subschemes over $B$ with relative multivariate Hilbert polynomial $\phi$. The natural map $\pi\colon\mathcal S\to B$ is flat. The fiber $\mathcal S_b$ over each $b\in B$ satisfies $h^0(\mathcal S_b,\mathcal O_{\mathcal S_b})=1$ only if $\mathcal S_b$ is connected. Theorem~\ref{chainmaps} yields the converse: if $\mathcal S_b$ is connected then $h^0(\mathcal S_b,\mathcal O_{\mathcal S_b})=1$. It follows that 
$U$ is also the subset parameterizing connected subschemes of $H_d^r(X)$, which is closed by the principle of connectedness. Furthermore, Theorem~\ref{chainmaps} yields, if the fibers of 
$\mathcal S$ are connected, that $\pi$ is a family of chains of rational curves and the natural map 
$\gamma\:\mathcal S\to T\times B$ is a family of chain maps. Since 
$J$ is isomorphic to an Abelian variety, it follows that the map $\mathcal S\to J$ factors through a map $B\to J$ yielding an invertible sheaf $\mathcal L$ on $X\times B$ of bidegree $(d,0)$ over $B$. Finally, the map $\mathcal S\to H^r_d(X)$ yields a locally free subsheaf 
$\V\subseteq p_{\mathcal S*}(\mathcal L(\gamma))$ of rank $r+1$ such that $\mathfrak g_{\mathcal S}:=(\L,\V)$ is a family of twisted generalized linear series. That these series are everywhere nonconstant follows from the fact that $\mathcal S\subseteq H_d^r(X)\times B$. That in addition they are everywhere invariant follows from the fact that $\mathcal S$ is $\cc$-invariant. Thus $\mathfrak g_{\mathcal S}$ is a family of continuous linear series over $B$. 

It is clear from the constructions that $\mathcal S_{\mathfrak g_{\mathcal S}}=\mathcal S$ and that $\mathfrak g_{\mathcal S_{\mathfrak g}}$ is equivalent to $\mathfrak g$. The last statement follows from the fact that $\text{\rm Hilb}_{H_d^r(X)}^{\phi,\cc}$ is projective because $\text{\rm Hilb}_{H_d^r(X)}^{\phi}$ is.  
\end{proof}

\section{Rigidity}

In this section we prove Theorem~\ref{chainmaps}. 

\begin{proof}[Proof of Theorem~\ref{chainmaps}] Since the coefficient of $v$ in $\phi$ is zero, we must have that 
$S_\red\subseteq H^r_d(X,L)$ for some invertible sheaf $L$ on $X$ with bidegree $(d,0)$, where ``red" denotes the reduced induced subscheme. Again, $H^r_d(X,L)$ is an invariant subscheme of $G\times T$, where $G:=\text{Grass}(r+1,W)$ for
$$
W:=W_1\oplus W_2:=\Gamma(Y,L_Y)\oplus\Gamma(Z,L_Z(dP)).
$$
Let $p_1$ and $p_2$ denote the projections of $G\times T$ onto the indicated factors.

Since $S$ is connected, so is $S_{\text{red}}$. Thus, since $S$ has linear multivariate polynomial, $S$ is of pure dimension~1, and hence so is $S_\red$. We will prove a series of claims:

\vskip0.2cm

{\bf First Claim:} {\it $S_{\text{\rm red}}$ has exactly one irreducible component $S_i$ satisfying $p_2(S_i)=T_i$ for each $i=0,\dots,d$, and $p_2$ restricts to an isomorphism $S_i\xrightarrow{\cong}T_i$. The remaining components of $S_{\text{\rm red}}$ are collapsed to the $N_\ell$ under $p_2$.}

\begin{proof}[Proof first claim] 
Since the coefficient of each $s_i$ in $\phi$ is $1$, 
we have $p_{2*}[S_\red]=[T]$. Thus, 
for each $i=0,\dots,d$ there exists a unique irreducible 
component $S_i$ of $S_{\text{red}}$
mapping birationally onto $T_i$ via $p_2$, the
remaining components being collapsed. Since the $T_i$ are rational, so must be the $S_i$, and the induced maps $S_i\to T_i$ are isomorphisms. 
Since $S$ is invariant, and the action of $\cc$ moves the points on
$T_i^*$ transitively for each $i$, it follows that each remaining
irreducible component of $S_\red$ collapses to a point among
$N_0,\dots,N_{d+1}$. 
\end{proof}

For each $i=0,\dots, d$, let $N_{i,0}$ be the unique point on $S_i$ mapped to $N_i$ and $N_{i,\infty}$ that mapped to $N_{i+1}$. For each $i=0,\dots,d+1$, let $S_{i-\frac{1}{2}}:=(S\cap p_2^{-1}(N_i))_\red$. 

\vskip0.2cm

{\bf Second Claim:} {\it The following statements hold:
\begin{enumerate}
\item The $S_{i-\frac{1}{2}}$ are pairwise disjoint.
\item For each $i=0,\dots,d+1$, we have that $S_{i-\frac{1}{2}}$ intersects $S_j$ if and only if $j\in\{i-1,i\}$, in which case $S_{i-\frac{1}{2}}\cap S_j$ is scheme-theoretically a point.
\item For each $i=1,\dots,d$ we have that $S_{i-\frac{1}{2}}$ is a point only if $N_{i-1,\infty}= N_{i,0}$, in which case that is the point on $S_{i-\frac{1}{2}}$.
\end{enumerate}
}

\begin{proof}[Proof second claim] The first statement is clear, as the images $p_2(S_{i-\frac{1}{2}})$ are pairwise disjoint. The third is clear as well, since, by definition, $N_{i-1,\infty}$ and $N_{i,0}$ are on $S_{i-\frac{1}{2}}$, if defined. As for the second, let  $i\in\{0,\dots,d+1\}$. Since $N_i\in T_j$ if and only if $j\in\{i-1,i\}$, we have that $S_{i-\frac{1}{2}}$ intersects $S_j$ if and only if $j\in\{i-1,i\}$. Furthermore, since each $S_j$ is mapped by $p_2$ isomorphically to $T_j$, and the scheme theoretic image of $S_{i-\frac{1}{2}}$ is $N_i$, we must have that $S_{i-\frac{1}{2}}\cap S_j$ is scheme-theoretically a point for each $j\in\{i-1,i\}$.
\end{proof}

{\bf Third Claim:} {\it $S_{i-\frac{1}{2}}$ is connected for each $i=0,\dots,d+1$.}

\begin{proof}[Proof third claim] Let $i\in\{0,\dots,d+1\}$. We will prove the claim through a series of statements.

\vskip0.1cm

{\bf Statement 1:} {\it For each $i=1,\dots,d$, every connected subcurve $C$ of $\overline{S_\red-\{S_{i-1}\cup S_i\}}$ intersecting both $S_{i-1}$ and $S_i$ is contained in $S_{i-\frac{1}{2}}$.}

Indeed, the image $p_2(C)$ is connected and contains points on $T_{i-1}$ and on $T_i$. It does not contain either $T_{i-1}$ or $T_i$ because only $S_{i-1}$ and $S_i$ cover these components, by First Claim. By the same reason, it does not contain points on $T_0\cup\cdots\cup T_{i-2}$ or $T_{i+1}\cup\cdots\cup T_d$. In particular, $p_2(C)$ contains neither $N_{i-1}$ nor $N_{i+1}$, and hence must be a point on both $T_{i-1}$ and $T_i$, thus equal to $N_i$. So $C\subseteq S_{i-\frac{1}{2}}$. 

\vskip0.1cm

{\bf Statement 2:} {\it $S_{i-\frac{1}{2}}$ contains at most one isolated point, which, if it exists, is equal to $N_{0,0}$ if $i=0$, to $N_{d,\infty}$ if $i=d+1$ and to $N_{i-1,\infty}= N_{i,0}$ if $1\leq i\leq d$.}

Indeed, let $Q$ be an isolated point on $S_{i-\frac{1}{2}}$. Then $Q$ belongs to an irreducible component $C$ of $S_\red$ whose image under $p_2$ contains $N_i$. Also, since $Q$ is isolated in $S_{i-\frac{1}{2}}$, we have that $C\not\subseteq S_{i-\frac{1}{2}}$, that is, $p_2(C)\neq N_i$. But then $C=S_{i-1}$ or $C=S_i$, whence $Q=N_{i-1,\infty}$ or $Q=N_{i,0}$, respectively. The former is possible only if $i>0$ and the latter is possible only if $i\leq d$. Hence, if $i=0$ then $C=S_0$ and $Q=N_{0,0}$, whereas if $i=d+1$ then $C=S_d$ and $Q=N_{d,\infty}$.

Assume now that $1\leq i\leq d$. Either $Q=N_{i-1,\infty}$ or $Q=N_{i,0}$. By contradiction, assume $N_{i-1,\infty}\neq N_{i,0}$. We will show directly that neither $N_{i-1,\infty}$ nor $N_{i,0}$ is isolated on $S_{i-\frac{1}{2}}$. Indeed, since $N_{i-1,\infty}\neq N_{i,0}$, we have $S_{i-1}\cap S_i=\emptyset$. Since $S$ is connected, there is a connected subcurve $C$ of $\overline{S_\red-\{S_{i-1}\cup S_i\}}$ intersecting both $S_{i-1}$ and $S_i$. Statement~1 yields $C\subseteq S_{i-\frac{1}{2}}$. Since $C$ intersects $S_i$, there is an irreducible component $C'$ of $C$ intersecting $S_i$. Since $C'\subseteq S_{i-\frac{1}{2}}$, the intersection is $N_{i,0}$, and hence $N_{i,0}$ is not isolated on $S_{i-\frac{1}{2}}$. Analogously, neither is $N_{i-1,\infty}$. 

\vskip0.1cm

{\bf Statement 3:} {\it Let $C$ be a maximal connected subcurve of $S_\red$ such that $C\subseteq S_{i-\frac{1}{2}}$. Then $C$ contains 
$N_{i-1,\infty}$ if $i\geq 1$ and $N_{i,0}$ if $i\leq d$.}

Indeed, reason by contradiction. Without loss of generality, assume $i\leq d$ and $C$ does not contain $N_{i,0}$. Since $p_2(C)=N_i$, this means that $C$ does not intersect $S_i$. Since $S_\red$ is connected, there is a connected subcurve $C'$ of $\overline{S_\red-\{S_i\cup C\}}$ intersecting both $S_i$ and $C$. Let $C''$ be an irreducible component of $C'$ intersecting $C$. Then $N_i\in p_2(C'')$. Since $C'\neq S_i$, either $p_2(C' )=N_i$ or $C''=S_{i-1}$, if $i\geq 1$. The former contradicts the maximality of $C$. Thus $i\geq 1$ and $C''=S_{i-1}$. Since $p_2(C)=N_i$, we must have that $S_{i-1}$ meets $C$ at $N_{i-1,\infty}$. Now, if $S_{i-1}$  meets $S_i$, it must be at a point over $N_i$, but the unique point on $S_{i-1}$ above $N_i$ is $N_{i-1,\infty}$. Since $N_{i-1,\infty}$ lies on $C$ and $C$ does not meet $S_i$, we must have that $S_{i-1}$ does not meet $S_i$. But since $C'$ meets $S_i$ and contains $S_{i-1}$, there is a connected subcurve $C'''\subseteq\overline{C'-S_{i-1}}$ meeting $S_{i-1}$ and $S_i$. By Statement~1, we must have $p_2(C''')=N_i$. Then $C'''$ does not intersect $C$ by the maximality of $C$. In particular, $C'''$ does not contain $N_{i-1,\infty}$. But $C'''$ intersects $S_{i-1}$, and can only intersect $S_{i-1}$ at $N_{i-1,\infty}$, a contradiction.

\vskip0.1cm

{\bf Statement 4:} {\it If $S_{i-\frac{1}{2}}$ is not a point, then 
$S_{i-\frac{1}{2}}$ has no isolated point.} 

Indeed, since $S_{i-\frac{1}{2}}$ has at most one isolated point by Statement~2, if $S_{i-\frac{1}{2}}$ is not a point then it contains a nonisolated point $Q$. Let $C$ be a maximal connected subcurve of $S_\red$ containing $Q$ such that $C\subseteq S_{i-\frac{1}{2}}$. Then Statement~3 says that $C$ contains $N_{i-1,\infty}$ if $i\geq 1$ and $N_{i,0}$ if $i\leq d$. In particular, neither $N_{i-1,\infty}$ for $i\geq 1$ nor $N_{i,0}$ for $i\leq d$ is isolated in $S_{i-\frac{1}{2}}$. But then $S_{i-\frac{1}{2}}$ contains no isolated point by Statement~2.

\vskip0.1cm

{\bf Statement 5:} {\it $S_{i-\frac{1}{2}}$ is connected.}

If $S_{i-\frac{1}{2}}$ is a point, then it is clearly connected. If not, then $S_{i-\frac{1}{2}}$ has no isolated point by Statement~4. It is thus a disjoint union of connected subcurves of $S_\red$. By Statement~3, each such subcurve contains $N_{i-1,\infty}$ if $i\geq 1$ and $N_{i,0}$ if $i\leq d$. But then all these subcurves contain the same point, and hence there is just one of them. In other words, $S_{i-\frac{1}{2}}$ is connected.
\end{proof}

{\bf Fourth Claim:} {\it There are irreducible components $C_0,\dots,C_m$ of $S_\red$ such that
\begin{enumerate} 
\item $C_0=S_0$ and $C_m=S_d$;
\item $C_{i-1}\cap C_i\neq\emptyset$ for $i=1,\dots,m$. 
\item There is an increasing sequence $i_0,i_1,\dots,i_d$ of integers such that $C_{i_\ell}=S_\ell$ for $\ell=0,\dots,d$;
\item For each $\ell=1,\dots,d$, we have $i_{\ell}=i_{\ell-1}+1$ if $S_{\ell-1}$ meets $S_\ell$; otherwise, $i_{\ell}>i_{\ell-1}+1$ and $C_{i_{\ell-1}+1}\cup\cdots\cup C_{i_{\ell}-1}$ is contained in $S_{\ell-\frac{1}{2}}$.
\end{enumerate}}

\begin{proof}[Proof fourth claim]
Start with $C_0:=S_0$. Each time $C_i=S_{\ell-1}$ for $\ell=1,\dots,d$, we do the following: If $S_{\ell}$ intersects $S_{\ell-1}$, we put $C_{i+1}:=S_\ell$. If not, then $N_{\ell-1,\infty}\neq N_{\ell,0}$, and thus $S_{\ell-\frac{1}{2}}$ is not a point by Second Claim. By Third Claim, $S_{\ell-\frac{1}{2}}$ is connected. By Second Claim, there are irreducible components $C$ and $C'$ of $S_{\ell-\frac{1}{2}}$ such that $C$ intersects $S_{\ell-1}$ and $C'$ intersects $S_\ell$. Put $C_{i+1}:=C$. Since $S_{\ell-\frac{1}{2}}$ is connected, there is a sequence of irreducible components of $S_{\ell-\frac{1}{2}}$, each intersecting the next, starting with $C_{i+1}$ and ending with $C'$. Denote these components sequentially by $C_{i+1},\cdots,C_{i+j}$, where $j$ is thus the number of components in the sequence and $C_{i+j}=C'$. Then put $C_{i+j+1}:=S_{\ell}$.
\end{proof} 

We observe that $S_\red$ may not be the union of the $C_i$. Indeed, 
$S_{\ell-\frac{1}{2}}$ may not be the union $C_{i_{\ell-1}+1}\cup\cdots\cup C_{i_{\ell}-1}$.

\vskip0.2cm

{\bf Fifth Claim:} {\it The $p_1(C_i)$ which are orbit closures are either points or smooth rational curves. The sum of their degrees in $G$ is at least $r+1$, with equality only if:
\begin{enumerate}
    \item all the $p_1(C_i)$ are orbit closures, 
    \item for $i,j\in\{0,\dots,m\}$ with $i<j$ such that $p_1(C_i)$ and $p_1(C_j)$ are curves that intersect, $p_1(C_i)\cap p_1(C_j)$ is scheme-theoretically a point, and $p_1(C_{i+1}),\dots,p_1(C_{j-1})$ are all equal to this point.
\end{enumerate}}

\begin{proof}[Proof fifth claim] For each $\ell=1,2$, let $\iota_\ell\:W_\ell\to W$ be the natural injection and $\rho_\ell\:W\to W_\ell$ the natural surjection. Since $S$ is $\cc$-invariant, each $C_i$ is $\cc$-invariant, whence the $p_1(C_i)$ are $\cc$-invariant. 

For each $i=0,\dots,m$, let $V_i\subseteq W$ be the image under $p_1$ of a general point on $C_i$. Put 
\[
a_i:=\min\{\dim_k\rho_1(V)\,|\,V\in p_1(C_i)\}
\quad\text{and}\quad 
A_i:=\max\{\dim_k\rho_1(V)\,|\,V\in p_1(C_i)\}.
\]
Thus $a_i\leq A_i$. If $V_i$ is fixed by the action of $\cc$, then $V_i=\rho_1(V_i)\oplus\rho_2(V_i)$. The same holds for each $V\in p_1(C_i)$ by continuity, whence $a_i=A_i$. In this case, if $p_1(C_i)$ is an orbit closure then $p_1(C_i)$ is the point $V_i$. If $V_i$ is not fixed, then $p_1(C_i)$ is the closure of the orbit of $V_i$ in $G$, and is thus a smooth rational curve of degree $A_i-a_i>0$ by Proposition~\ref{hV}. Furthermore, the boundary of the orbit of $V_i$ consists of two points on $p_1(C_i)$ where $a_i$ and $A_i$ are attained. 

Now, for each $i=1,\dots,m$, since $C_{i-1}\cap C_i\neq\emptyset$, it follows that  
the intervals $[a_{i-1},A_{i-1}]$ and $[a_i,A_i]$ intersect, and thus 
\[
I:=\bigcup_{i=0}^m[a_i,A_i]
\]
is an interval contained in $[0,r+1]$.
Now, $C_0=S_0$ and $C_m=S_d$. Since $p_2(N_{0,0})=N_0$ and $p_2(N_{d,\infty})=N_\infty$, we have that
\[
p_1(N_{0,0})\subseteq\Gamma(Y,L|_Y)\oplus\Gamma(Z,L|_Z(-P))\quad
\text{and}\quad 
p_1(N_{d,\infty})\subseteq\Gamma(Y,L|_Y(-(d+1)P))\oplus\Gamma(Z,L|_Z(dP)).
\]
Since $L|_Z(-P)$ and $L|_Y(-(d+1)P)$ have negative degrees, it follows that $A_0=r+1$ and $a_m=0$. 
Thus $I=[0,r+1]$. Then the sum of the degrees of the $p_1(C_i)$ for $i$ such that $V_i$ is nonfixed, which is $\sum_i(A_i-a_i)$, is at least $r+1$.

Assume the sum is exactly $r+1$. Then, first, if $V_i$ is fixed, $p_1(C_i)$ must be a point, $V_i$. Thus, all the $p_1(C_i)$ are orbit closures. Second, the intervals $[a_i,A_i]$ can overlap only at the boundaries. For each $i=1,\dots,m$, since $[a_{i-1},A_{i-1}]$ and $[a_i,A_i]$ intersect, we must have that either 
$a_i=A_{i-1}$ or $A_i=a_{i-1}$. Since $A_0=r+1$ and $a_m=0$, we must have that $A_i=a_{i-1}$ for each $i=1,\dots,m$. 

Let now $i,j\in\{0,\dots,m\}$ with $i<j$ such that $p_1(C_i)$ and $p_1(C_j)$ are curves that intersect. Since $i<j$, we have $A_j\leq a_i$. Since $p_1(C_i)$ and $p_1(C_j)$ intersect, the intervals $[a_i,A_i]$ and $[a_j,A_j]$ intersect as well, and hence $A_j=a_{j-1}=A_{j-1}=\cdots=A_{i+1}=a_i$. It follows that 
$p_1(C_{i+1}),\dots,p_1(C_{j-1})$ are points, actually, the same point $V\in G$, since $C_{i+1}\cup\cdots\cup C_{j-1}$ is connected.  Furthermore, since $A_j=a_i$, and $p_1(C_i)$ and $p_1(C_j)$ are orbit closures that intersect, it follows from Proposition~\ref{hV} that 
there is a unique point of intersection, which is
\[
\iota_1^{-1}(V_i)\oplus\rho_2(V_i)=\rho_1(V_j)\oplus\iota_2^{-1}(V_j).
\]
In addition, Proposition~\ref{hVV} yields that the intersection is transverse. Since 
$C_i$ intersects $C_{i+1}$ and $C_j$ intersects $C_{j-1}$, we must have that $V$ is the intersection.
\end{proof}

Put $S^0:=C_0\cup\dots\cup C_m$. 

\vskip0.2cm

{\bf Sixth Claim:} {\it The sum of the degrees of the $p_1(C_i)$ in $G$ is $r+1$. Each $C_i$ is rational and the map $C_i\to p_1(C_i)$ is either an isomorphism or constant, the latter only if $i=i_\ell$ for some $\ell$. Furthermore, $S^0=S_\red$ and $S$ is generically reduced. In particular, $S_{-\frac{1}{2}}$ and $S_{d+\frac{1}{2}}$ are points.}

\begin{proof}[Proof sixth claim] Since $S$ has Hilbert polynomial $\phi$, and since the coefficient of $u$ in $\phi$ is $r+1$, the products of the degree of the covering $C_i\to p_1(C_i)$ with the degree of $p_1(C_i)$ sum up to at most $r+1$. On the other hand, the degrees of the $p_1(C_i)$ sum up to at least $r+1$ by Fifth Claim. It thus follows that both sums are $r+1$ and that $C_i\to p_1(C_i)$ is either constant or an isomorphism for each $i=0,\dots,m$. However, since $C_i\subseteq G\times T$, the curve $p_1(C_i)$ can only be a point if $p_2(C_i)$ is not one, that is, if $i=i_\ell$ for some $\ell$, in which case $C_i$ is isomorphic to $T_\ell$ under $p_2$. In any case, $C_i$ is a rational curve. 

Finally, the multivariate Hilbert polynomial of $S^0$ has the same linear part as that of $S$. Since $S^0\subseteq S_\red\subseteq S$, it follows that $S^0=S_\red$ and $S$ is generically reduced. 
\end{proof}

\vskip0.2cm

{\bf Seventh Claim:} {\it For each $i,j\in\{0,\dots,m\}$ with $i<j$, the curve $C_i$ intersects $C_j$ only if $j=i+1$, transversally at a single point if so.}

\begin{proof}[Proof seventh claim] Assume $C_i\cap C_j\neq\emptyset$. 
Since the $S_{\ell-\frac{1}{2}}$ are disjoint, and $S_h\cap S_q\neq\emptyset$ only if $|h-q|\leq 1$, we must have 
$i_{\ell-1}\leq i<j\leq i_{\ell}$ for some $\ell$. Now, since $p_2(S_{\ell-\frac{1}{2}})=N_\ell$ scheme-theoretically, and 
$S_{\ell-1}$ is mapped isomorphically to 
$T_{\ell-1}$, 
we have that $C_j$ intersects $C_{i_{\ell-1}}$ only if $j=i_{\ell-1}+1$, transversally at $N_{\ell-1,\infty}$ if so. Thus 
the claim holds if $i=i_{\ell-1}$. By analogy, it holds as well if $j=i_\ell$.

Assume $i_{\ell-1}<i<j<i_\ell$. Then the maps $C_q\to p_1(C_q)$ are isomorphisms for all $q$ with $i\leq q\leq j$ by Sixth Claim.
Since $C_i\cap C_j\neq\emptyset$, it follows from Fifth Claim that $j=i+1$ and $p_1(C_i)\cap p_1(C_j)$ is scheme-theoretically a point. Furthermore, since the maps $C_i\to p_1(C_i)$ and $C_j\to p_1(C_j)$ are isomorphisms by Sixth Claim, also $C_i\cap C_j$ is scheme-theoretically a point.
\end{proof}

{\bf Eighth Claim:} {\it The curve $S$ is a chain of rational curves and $p_1\:S\to T$ is a chain map.}

\begin{proof} It follows from Sixth Claim that the $C_i$ are rational, and thus $S^0$ is a chain of rational curves by Seventh Claim. Then $S^0$ is connected of arithmetic genus $0$, and hence the constant part of its multivariate Hilbert polynomial coincides with that of $S$. So does its linear part by Sixth Claim. Hence $S$ and $S^0$ have the same Hilbert polynomial, and since $S^0\subseteq S$, we must have $S^0=S$. So $S$ is a chain of rational curves. First and Fourth Claims now yield that $p_1\colon S\to T$ is a chain map.
\end{proof}

The proof of Theorem~\ref{chainmaps} is now complete.
\end{proof}

\section{Level-$\delta$ limit linear series}\label{deltalls}

Let $\delta=(\delta_1,\dots,\delta_d)$ be a $d$-tuple of natural numbers. Let 
$$
\Delta:=\Delta(\delta):=\Big\{0,\frac{1}{\delta_1},\frac{2}{\delta_1},\dots,1,
1+\frac{1}{\delta_2},1+\frac{2}{\delta_2},\dots,2,\dots,d-1,
d-1+\frac{1}{\delta_d},d-1+\frac{2}{\delta_d},\dots,d\Big\}.
$$
We say that $i,j\in\Delta$ are \emph{consecutive} if $i<j$ and $\Delta\cap(i,j)=\emptyset$.

\subsection{Review}

The reader will recognize the notions and definitions below from \cite{ENR}, where we restricted to the case where all the $\delta_i$ are equal. However, the general theory is similar.

Let $L$ be an invertible sheaf on $X$. The \emph{twist $\delta$-sequence} associated to $L$ is the collection of sheaves
$L^{(i)}$ indexed by
$i\in\Delta$, where
$L^{(i)}:=L\otimes\mathcal O_X^{(i)}$ if $i\in\mathbf Z$ and, otherwise,
\begin{equation}\label{Linot}
L^{(i)}:=\Big(L|_Y\ox\O_Y\big(\lfloor-i\rfloor P\big)\Big)\bigoplus\Big(L|_Z\ox\O_Z\big(\lfloor i\rfloor P\big)\Big),
\end{equation}
where $\lfloor i\rfloor$ stands for the floor of $i$, the largest integer at most $i$.

The condition on the $L^{(i)}$ ensures that there are isomorphisms
$$
L^{(i)}|_Z\longrightarrow\text{Ker}\big(L^{(j)}\to L^{(j)}|_Y\big)\quad\text{and}\quad
L^{(j)}|_Y\longrightarrow\text{Ker}\big(L^{(i)}\to L^{(i)}|_Z\big)
$$
for each consecutive $i,j\in\Delta$, whence there are nonzero maps $\varphi^i\colon L^{(i)}\to L^{(j)}$ factoring through $L^{(i)}|_Z$ and $\varphi_i\colon L^{(j)}\to L^{(i)}$ factoring through $L^{(j)}|_Y$. The 
maps $\varphi^i$ and $\varphi_i$ are defined up to multiplication by $k^*$. 

If $L=\mathcal O_X$, then our notation is compatible with that used in Subsection~\ref{sec:twisters}. 

\begin{defn} 
A {\it level-$\delta$ limit linear series} of degree $d$ on $X$ is
the data $\mathfrak g=(L;V^{(i)},\,i\in\Delta)$ of an
invertible sheaf $L$ on $X$ of bidegree $(d,0)$ and vector subspaces
$V^{(i)}\subseteq\Gamma(X,L^{(i)})$ for $i\in\Delta$
of equal dimension such that
\begin{equation}\label{compatible}
\varphi^i(V^{(i)})\subseteq
V^{(j)}\text{\ and\ }
\varphi_i(V^{(j)})\subseteq
V^{(i)}\text{\ for each consecutive }i,j\in\Delta.
\end{equation}
We say that $\mathfrak g$ has rank $r$
if the $V^{(i)}$ have projective dimension $r$. We say that
$\mathfrak g$ is \emph{exact} if 
\begin{equation}\label{exact}
\text{Im}\big(\varphi^i|_{V^{(i)}}\big)=
\text{Ker}\big(\varphi_i|_{V^{(j)}}\big)
\text{\ and\ }
\text{Im}\big(\varphi_i|_{V^{(j)}}\big)=
\text{Ker}\big(\varphi^i|_{V^{(i)}}\big)\text{\ for consecutive }i,j\in\Delta.
\end{equation}
\end{defn}


Let $\mathfrak g=(L;V^{(i)},\,i\in\Delta)$ be a level-$\delta$ limit linear series of rank $r$ and degree $d$ on $X$. It follows from \eqref{Linot} that 
$L^{(i)}\subseteq L|_Y\oplus L|_Z(dP)$ for each $i\in\Delta$. Thus, subspaces $V^{(i)}$ of $\Gamma(X, L^{(i)})$ may be viewed as subspaces of 
$U:=U_1\oplus U_2$, where $U_1:=\Gamma(Y,L|_Y)$ and $U_2:=\Gamma(Z,L|_Z(dP))$. As such, Conditions~\eqref{compatible} are equivalent to
\begin{equation}\label{compatible2}
\rho_2(V^{(i)})\subseteq\iota_2^{-1}(V^{(j)})
\quad\text{and}\quad
\rho_1(V^{(j)})\subseteq\iota_1^{-1}(V^{(i)})
\quad\text{for each consecutive }i,j\in\Delta,
\end{equation}
whereas Conditions~\eqref{exact} are equivalent to
\begin{equation}\label{exact2}
\rho_2(V^{(i)})=\iota_2^{-1}(V^{(j)})
\quad\text{and}\quad
\rho_1(V^{(j)})=\iota_1^{-1}(V^{(i)})
\quad\text{for each consecutive }i,j\in\Delta,
\end{equation}
where 
$\iota_j\: U_j\to U_1\oplus U_2$ is the natural injection and 
$\rho_j\: U_1\oplus U_2\to U_j$ is the natural surjection for $j=1,2$.

The data of the vector spaces $(V^{(i)},\,i\in\Delta)$ and the maps
$h_i:=\varphi_{i}|_{V^{(j)}}$ and 
$h^i:=\varphi^i|_{V^{(i)}}$ for consecutive $i,j\in\Delta$ 
will be called a \emph{linked chain of vector spaces}. For each $i\in\Delta$, 
let $m_i:=(r+1)-p_i-q_i$ where
$$
p_i:=\dim\Ker(h_{i-1}),\quad q_i:=\dim\Ker(h^i).
$$
Let $N_{\mathfrak g}\colon\Delta\to\mathbf Z^3$ be the map taking $i$ to $(p_i,q_i,m_i)$; it is called the \emph{numerical data} associated to the linked chain and thus to $\mathfrak g$. We say $N_{\mathfrak g}$ is \emph{exact} if $\sum m_i=r+1$. Then $\mathfrak g$ is exact if and only 
if $N_{\mathfrak g}$ is; see \cite[Prop. 3.6, p.~836]{ENR}. We say $N_{\mathfrak g}$ and $\mathfrak g$ are \emph{minimal} if $m_i>0$ for each $i\in\Delta-\mathbf Z$.

As in \cite{ENR}, Prop.~3.2, p.~834, there are a projective
scheme $G_{d,\delta}^r(X)$ parameterizing level-$\delta$
limit linear series $\mathfrak g=(L;V^{(i)},\,i\in\Delta)$ of degree $d$ and rank $r$ on $X$, and an open
subscheme
$$
G^{r,*}_{d,\delta}(X)\subseteq
G^{r}_{d,\delta}(X),
$$
parameterizing those which are exact. Also, there is a natural action of 
$\cc^\Delta$ on $G_{d,\delta}^r(X)$ such that 
$(c_i,i\in\Delta)\ast(L_1;V_1^{(i)},i\in\Delta)=(L_2; V_2^{(i)},i\in\Delta)$ if and only if $L_1\cong L_2$ and $V_2^{(i)}=V_1^{(i)}$ for $i\in\Delta\cap\mathbf Z$ and $V_2^{(i)}=(c_i,1)V_1^{(i)}$ otherwise, as subspaces of $U=U_1\oplus U_2$. The action leaves the sets of minimal and exact limit linear series invariant. Two level-$\delta$ limit linear series 
$(L_1;V_1^{(i)},i\in\Delta)$ and $(L_2;V_2^{(i)},i\in\Delta)$ are called \emph{equivalent} if they lie on the same orbit by the action of $\cc^{\Delta}$. Equivalent limit linear series have the same numerical data.

As shown in \cite{ENR}, Section~3, the numerical data yield a stratification of $G^r_{d,\delta}(X)$ into locally closed subsets, as follows: For each 
function $N\:\Delta\to\z^3$, let
\[
G_{d,\delta}^r(X;N):=\left\{\mathfrak{g}\in 
G_{d,\delta}^r(X)\,|\, N_{\mathfrak g}=N\right\}.
\]
So that $G_{d,\delta}^r(X;N)$ is nonempty, a certain condition on $N$ is necessary, called admissibility in \cite{ENR}, Def.~3.7, p.~836. In addition, the open
subscheme $G^{r,*}_{d,\delta}(X)\subseteq G^{r}_{d,\delta}(X)$ decomposes as a disjoint union of subsets which are both open and closed,
\begin{equation}\label{estrata}
G_{d,\delta}^{r,*}(X)=\coprod_{N\text{ exact}}
G_{d,\delta}^r(X;N).
\end{equation}
The $G_{d,\delta}^r(X;N)$ can be given a natural structure of moduli scheme.

For each (exact) function $N\:\Delta\to\mathbf Z^3$, let $\Delta_N\subseteq\Delta$ be the subset of those $i$ for which $i\in\mathbf Z$ or $N(i)\not\in\mathbf Z^2\times\{0\}$. Let $\delta'=(\delta'_1,\dots,\delta'_d)$ be the unique $d$-tuple of natural numbers for which there is an isomorphism of posets $f\:\Delta'\to\Delta_N$ such that $f(i)=i$ for each $i\in\z$, 
where $\Delta':=\Delta(\delta')$. Let $N':=Nf$. Then $N'$ is (exact and) minimal. Analogously to what was shown in \cite{ENR}, Prop.~4.3, p.~839, the natural forgetful morphism $G_{d,\delta}^r(X;N)\to G_{d,\delta'}^r(X;N')$, taking the limit linear series $(L;V(i),i\in\Delta)$ to $(L;V(f(i)),i\in\Delta')$, is an isomorphism.

Because of this isomorphism, we will consider the open subscheme
\[
G_{d,\delta}^{r,*}(X;\text{\rm min})\subseteq G_{d,\delta}^{r,*}(X),
\]
which is the (disjoint) union of the 
$G_{d,\delta}^r(X;N)$ for $N$ exact and minimal. Clearly, we have
$G_{d,\mathbf 1}^{r,*}(X;\text{\rm min})=G_{d,\mathbf 1}^{r,*}(X)$, where $\mathbf 1:=(1,\dots,1)$.

\subsection{Second main theorem}

Let $T^{(\delta)}$ be the chain of rational curves obtained from $T$ by splitting the branches of $T$ at $N_i$ and connecting them by a chain of rational curves of length $\delta_i-1$ for each $i=1,\dots,d$. Thus we index the components $T^{(\delta)}_i$ of $T^{(\delta)}$ with $i\in\Delta$. For each consecutive  $i,j\in\Delta$, let $N^{(\delta)}_j$ be the intersection of $T^{(\delta)}_{i}$ with $T^{(\delta)}_j$. Clearly, there is a natural map $\mu^{(\delta)}\:T^{(\delta)}\to T$ collapsing all the components $T^{(\delta)}_{i}$ for $i\not\in\mathbf Z$, and identifying $T^{(\delta)}_{i}$ with $T_i$ for $i\in\mathbf Z$.

Let $N_0^{(\delta)}$ and $N_\infty^{(\delta)}$ be the points on $T^{(\delta)}$ such that $\mu^{(\delta)}(N_0^{(\delta)})=N_0$ and $\mu^{(\delta)}(N_\infty^{(\delta)})=N_\infty$. Then $(T^{(\delta)},N_0^{(\delta)},N_\infty^{(\delta)})$ is a two-pointed chain and $\mu^{(\delta)}$ is a chain map. 

For each $i\in\mathbf Z$, let $E_i^{(\delta)}$ be the point on $T_i^{(\delta)}$ above $E_i$. For each $i\in\Delta-\mathbf Z$, fix $E_i^{(\delta)}\in(T^{(\delta)}_i)^*$. The chain action of $\cc$ on $(T,N_0,N_\infty)$ such that $1\ast E_i=E_i$ for each $i=0,\dots,d$ lifts to a unique chain action on $(T^{(\delta)},N_0^{(\delta)},N_\infty^{(\delta)})$ such that $1\ast E^{(\delta)}_i=E^{(\delta)}_i$ for each $i\in\Delta$.

\begin{thm}\label{allevels} For each 
$\delta=(\delta_1,\dots,\delta_d)\in\mathbf N^d$ and exact and minimal function $N\colon\Delta(\delta)\to\n^3$
there is a natural morphism
\[
\Psi_\delta^N\:G^{r}_{d,\delta}(X;N)\longrightarrow
G_d^r(X).
\]
The map takes a (minimal exact) level-$\delta$ limit linear series $(L;V^{(i)},i\in\Delta)$ on $X$ with numerical data $N$ to the continuous linear series $(L,\V)$ along $\mu^{(\delta)}$, where $\V|_{E_i^{(\delta)}}$ is taken to $V^{(i)}$ under the choice of an isomorphism $L(\mu^{(\delta)})|_{X\times E_i^{(\delta)}}\xrightarrow{\cong} L^{(i)}$ for each $i\in\Delta$. Each nonempty fiber of $\Psi_\delta^N$ is an equivalence class of level-$\delta$ limit linear series, and vice-versa. Finally, the images of all the $\Psi_\delta^N$ form a partition of the underlying set of $G_d^r(X)$. 
\end{thm}

(If $\delta=\mathbf 1$, then $\mu^{(\delta)}=1_T$, all the $L^{(i)}$ are invertible, and thus the $\V|_{E_i}$ do not depend on the choice of an isomorphism $L(1_T)|_{X\times E_i}\cong L^{(i)}$. Also, $\Psi^N_{\mathbf 1}$ is injective.)


\begin{proof} Put $\Delta=\Delta(\delta)$. Consider a family of level-$\delta$ limit linear series on $X$ of degree $d$, rank $r$ and numerical data $N$ over a scheme $B$. 
The family is given by an invertible sheaf $\L$ of relative bidegree $(d,0)$ on $X\times B/B$ and a family of generalized linear series $(\L^{(i)},\mathcal V^{(i)})$ on $X\times B/B$, where $\L^{(i)}:=\L(\mu^{(\delta)})|_{X\times E_i^{(\delta)}\times B}$ under the usual identification of $E_i^{(\delta)}\times B$ with $B$, for each $i\in\Delta$. The data $(\L;\V^{(i)},i\in\Delta)$ satisfies certain conditions, to be explained next. 

Let $D$ be an effective divisor of the smooth locus of $X$. Put
\[
\mathcal W_1:=p_{B*}\big(\L|_{Y\times B}((D\cap Y)\times B)\big)
\quad\text{and}\quad
\mathcal W_2:=p_{B*}\big(\L|_{Z\times B}(dP\times B+(D\cap Z)\times B)\big),
\]
and let $\W:=\W_1\oplus\W_2$, where $p_B\:X\times B\to B$ is the projection. Assume $D$ is ample enough that the $\mathcal W_i$ be locally free with formation commuting with base change, and the natural maps $\V^{(i)}\to\mathcal W_1\oplus\mathcal W_2$ be injective with locally free quotients. 

Let $\rho_\ell\:\W_1\oplus\W_2\to\W_\ell$ and $\iota_\ell\:\W_\ell\to\W_1\oplus\W_2$ be the natural maps for $\ell=1,2$. 
The conditions $(\L;\V^{(i)},i\in\Delta)$ satisfies are three. First, the map $\V^{(i)}\to \mathcal W_\ell$ induced by $\rho_\ell$ has locally free cokernel for each $i\in\Delta$ and $\ell=1,2$. It follows that 
$\mathcal A^{(i)}_\ell:=\rho_\ell(\V^{(i)})$ and 
$\mathcal B^{(i)}_\ell:=\iota_{\ell}^{-1}(\V^{(i)})$ 
are locally free. Their ranks are specified by $N$, the second condition.
Notice that
\[
\mathcal B_1^{(i)}\oplus\mathcal B_2^{(i)}\subseteq
\V^{(i)}
\subseteq\mathcal A_1^{(i)}\oplus\mathcal A_2^{(i)}
\subseteq \mathcal W_1\oplus\mathcal W_2,
\]
and the induced map
\[
\lambda^{(i)}_\ell\:\frac{\V^{(i)}}{\mathcal B_1^{(i)}\oplus\mathcal B_2^{(i)}}\xrightarrow{\quad\cong\quad}\frac{\mathcal A_\ell^{(i)}}{\mathcal B_\ell^{(i)}}
\]
is an isomorphism for each $\ell=1,2$. Composing the inverse of $\lambda^{(i)}_1$ with $\lambda^{(i)}_2$ gives us an isomorphism
\[
\zeta^{(i)}\: \frac{\mathcal A_1^{(i)}}{\mathcal B_1^{(i)}}\xrightarrow{\quad\cong\quad}\frac{\mathcal A_2^{(i)}}{\mathcal B_2^{(i)}}.
\]
Finally, $\mathcal A^{(i)}_2=\mathcal B^{(j)}_2$ and $\mathcal A^{(j)}_1=\mathcal B^{(i)}_1$ 
for each consecutive $i,j\in\Delta$, by the exactness of $N$, the third condition.

For each $i\in\Delta$, let $p_i\: T^{(\delta)}_i\times B\to B$ and $q_i\:T^{(\delta)}_i\times B\to T^{(\delta)}_i$ be the projections, and let $\widetilde\V^{(i)}\subseteq p_i^*(\mathcal A_1^{(i)}\oplus\mathcal A_2^{(i)})$ be the subsheaf whose quotient is the cokernel of the composition 
\[
q_i^*\mathcal O_{T^{(\delta)}_i}(-1)\bigotimes\frac{p_i^*\mathcal A_1^{(i)}}{p_i^*\mathcal B_1^{(i)}} \xrightarrow{\quad q_i^*\xi_i\otimes 1\quad}
\frac{p_i^*\mathcal A_1^{(i)}}{p_i^*\mathcal B_1^{(i)}}\bigoplus\frac{p_i^*\mathcal A_1^{(i)}}{p_i^*\mathcal B_1^{(i)}} \xrightarrow{\quad 1\oplus p_i^*\zeta^{(i)}\quad}
\frac{p_i^*\mathcal A_1^{(i)}}{p_i^*\mathcal B_1^{(i)}}\bigoplus\frac{p_i^*\mathcal A_2^{(i)}}{p_i^*\mathcal B_2^{(i)}},
\]
where $\xi_i$ is the inclusion $\mathcal O_{T^{(\delta)}_i}(-1)\hookrightarrow\mathcal O_{T^{(\delta)}_i}^2$ of the tautological subsheaf associated to an isomorphism $T^{(\delta)}_i\xrightarrow{\cong}\mathbf P^1_k$ such that 
\begin{equation}\label{cvary}
\mathcal O_{T^{(\delta)}_i}(-1)\big|_{c\ast E_i^{(\delta)}}=(c^{-1},1)\mathcal O_{T^{(\delta)}_i}(-1)\big|_{E_i^{(\delta)}}\quad\text{for each }c\in\cc(k).
\end{equation}
Then $\widetilde\V^{(i)}\subseteq p_i^*\W$ with locally free quotient, and satisfies
\begin{equation}\label{eq:VtildeV}
\widetilde\V^{(i)}\big|_{c\ast E_i^{(\delta)}}=(c^{-1},1)p_1^*\V^{(i)}\big|_{E_i^{(\delta)}}\quad\text{for each }c\in\cc(k).
\end{equation}

Moreover, there is a subsheaf $\widetilde\V\subseteq p^*\W$ with locally free quotient such that $\widetilde\V|_{T^{(\delta)}_i\times B}=\widetilde\V^{(i)}$ for each $i\in\Delta$, where $p\:T^{(\delta)}\times B\to B$ is the projection, because 
\[
\widetilde\V^{(i)}|_{N^{(\delta)}_j\times B}=\mathcal B_1^{(i)}\oplus\mathcal A_2^{(i)}=\mathcal A_1^{(j)}\oplus\mathcal B_2^{(j)}=\widetilde\V^{(j)}|_{N^{(\delta)}_j\times B}
\]
for each consecutive $i,j\in\Delta$.

Finally, $\widetilde\V\subseteq\widetilde p_*(\L(\mu^{(\delta)}))$, where $\widetilde p\: X\times T^{(\delta)}\times B\to T^{(\delta)}\times B$ is the projection. Indeed, 
\[
\widetilde p_*\big(\L(\mu^{(\delta)})\big)\subseteq\widetilde p_*\big(\L(D\times B)(\mu^{(\delta)})\big)\subseteq p^*\W
\]
with locally free quotients, if $D$ is sufficiently ample. Since $\V^{(i)}\subseteq p_{B*}\L^{(i)}$, it follows from \eqref{eq:VtildeV} that 
$\widetilde\V\subseteq\widetilde p_*(\L(D\times B)(\mu^{(\delta)}))$ over a dense open subscheme of $T^{(\delta)}\times B$, and then over the whole scheme. Also, the composition 
\[
\widetilde\V\longrightarrow\widetilde p_*\big(\L(D\times B)(\mu^{(\delta)})\big)\longrightarrow
\widetilde p_*\big(\L(D\times B)(\mu^{(\delta)})\big|_{D\times T^{(\delta)}\times B}\big)
\]
is zero because it restricts to zero on a dense open subscheme of $T^{(\delta)}\times B$, thus showing that $\widetilde\V\subseteq\widetilde p_*(\L(\mu^{(\delta)}))$. 

Thus we get a family of twisted generalized linear series $(\L,\widetilde\V)$ of degree $d$ and rank $r$ on $X$ along $\mu^{(\delta)}\times 1_B$. From the construction, $(\L,\widetilde\V)$ is clearly everywhere invariant. Furthermore, it is everywhere nonconstant because $N$ is minimal. So
$(\L,\widetilde\V)$ is a family of continuous linear series along the family of chain maps $1_B\times\mu^{(\delta)}$. It is clear that the association of $(\L,\widetilde\V)$ to $(\L;\V^{(i)},i\in\Delta)$ commutes with base change, thus defining the morphism 
$\Psi_\delta^N$ in the statement of the theorem.

The action by $\cc^{\Delta}$ on $G^{r,*}_{d,\delta}(X;N)$ changes $(\L,\widetilde\V)$ to an equivalent continuous linear series, whence 
$\Psi_\delta^N$ is $\cc^{\Delta}$-invariant. Furthermore, the equivalence class of $(\L;\V^{(i)},i\in\Delta)$ can be recovered from $(\L,\widetilde\V)$, as 
$\L(\mu^{(\delta)})|_{X\times E_i^{(\delta)}\times B}\cong\L^{(i)}$ for each $i\in\Delta$ and, under a certain choice of isomorphism, $\V^{(i)}$ is the image of $\widetilde\V|_{E^{(\delta)}_i\times B}$. It follows that each nonempty fiber of $\Psi_\delta^N$ is an equivalence class of 
level-$\delta$ limit linear series, as stated in the theorem.

Finally, let $(L,\V)$ be a continuous linear series of degree $d$ and rank $r$ on $X$ along a chain map $\mu\:S\to T$. For each $i=1,\dots,d$, let 
$\delta_i$ be the number of irreducible components of $\mu^{-1}(N_i)$. 
Put $\delta:=(\delta_1,\dots,\delta_d)$ and $\Delta:=\Delta(\delta)$. 
There is a unique isomorphism between the poset of irreducible components of $S$ and $\Delta$. For each $i\in\Delta$, let $S_i$ be the corresponding component. Also, let $N_j:=S_i\cap S_j$ for each consecutive $i,j\in\Delta$. 

For each $i\in\Delta$, fix a point $E'_i\in S^*_i$. Put 
$V^{(i)}:=\V|_{E'_i}$ and 
$L^{(i)}:=L(\mu)|_{X\times E'_i}$. Then 
$(L^{(i)},V^{(i)})$ is a linear series on $X$.

Observe that the $L^{(i)}$ form the twist $\delta$-sequence associated 
to $L$. Also, let $U_1:=\Gamma(Y,L|_Y)$ and 
$U_2:=\Gamma(Z,L|_Z(dP))$, and put $U:=U_1\oplus U_2$. Since $(L,\V)$ is everywhere invariant,
$$
\V_{N_j}=\lim_{c\to\infty}(c^{-1},1)\V|_{E'_i}=
\lim_{c\to 0}(c^{-1},1)\V|_{E'_j}
$$
as subspaces of $U$ for each 
consecutive $i,j\in\Delta$. But Proposition~\ref{hV} yields
$$
\lim_{c\to\infty}(c^{-1},1)\V|_{E'_i}=
\iota_1^{-1}(\V|_{E'_i})\oplus\rho_2(\V_{E'_i})
\text{\  \ and \  \ }
\lim_{c\to 0}(c^{-1},1)\V|_{E'_j}=
\rho_1(\V|_{E'_j})\oplus\iota_2^{-1}(\V|_{E'_j}),
$$
where $\iota_\ell\: U_\ell\to U$ is the natural injection and 
$\rho_\ell\: U\to U_\ell$ is the natural surjection for $\ell=1,2$. 
Thus the $V^{(i)}$ satisfy Conditions~\eqref{exact2}, 
and hence $(L,\V^{(i)},i\in\Delta)$ is an exact 
level-$\delta$ limit linear series on $X$. Finally, since $(L,\V)$ is everywhere nonconstant, $(L,\V^{(i)},i\in\Delta)$ is minimal, and is thus parameterized by a point on $G^{r}_{d,\delta}(X;N)$ for a unique exact, minimal $N\:\Delta\to\mathbf Z^3$. 

Since $(L,\V)$ is everywhere invariant, and since \eqref{cvary} holds for each $i\in\Delta$, it follows that there is a $T$-isomorphism $T^\Delta\to S$ taking $E_i^\Delta$ to $E'_i$ for each $i\in\Delta$ and inducing an equivalence between $(L,\widetilde\V)$, the continuous linear series corresponding to $(L,\V^{(i)},i\in\Delta)$ under $\Psi_\delta^N$, and $(L,\V)$. 
\end{proof}

\begin{prop}\label{toOss} There is a natural $J$-map $\Phi\: G^r_d(X)\to G^r_{d,\mathbf 1}(X)$ such that $\Phi\Psi_{\mathbf 1}^N$ is the open embedding of $G^{r}_{d,\mathbf 1}(X;N)$ in $G^r_{d,\mathbf 1}(X)$ for each exact, minimal $N$.
\end{prop}

\begin{proof} Let $(\L,\V)$ be a family of continuous linear series along a family of chain maps $\gamma\:\mathcal S\to T\times B$. Then $\gamma$ restricts to an isomorphism $\gamma^{-1}(E_i\times B)\to E_i\times B$ for each $i=0,\dots,d$. Under the natural identification of $E_i\times B$ with $B$, the restriction of 
$\L(\gamma)$ over $\gamma^{-1}(E_i\times B)$ is thus isomorphic to
$\L^{(i)}:=\L\otimes p_1^*\mathcal O_X^{(i)}$, where $p_1\:X\times B\to X$ is the projection, and the restriction of $\V$ to $E_i\times B$ is 
thus isomorphic to a subsheaf $\V^{(i)}$ of $p_{2*}\L^{(i)}$ such that 
$(\L^{(i)},\V^{(i)})$ is a family of linear series on $X\times B/B$, where $p_2\:X\times B\to B$ is the projection. The compatibility of the $\V^{(i)}$, that is, the fact that the image of $\V^{(i)}$ under the map induced by the natural map $\mathcal O_X^{(i)}\to\mathcal O_X^{(i\pm 1)}$ is contained in
$\V^{(i\pm 1)}$ for each $i$, follows from the argument given in the proof of the last statement of Theorem~\ref{allevels}. Thus we have a $B$-point of $G^{r}_{d,\mathbf 1}(X)$. The association commutes with base change, whence we obtain the $J$-map $\Phi$. Its property is a consequence of the construction of $\Psi_{\mathbf 1}^N$ in the proof of Theorem~\ref{allevels}.
\end{proof}

\bibliographystyle{plain}

\end{document}